\documentclass{amsart}

\usepackage{amsmath,amssymb,amsfonts,graphics,color}
\usepackage{epsfig}
\usepackage{latexsym}
\usepackage{euscript}
\title{Uniform linear embeddings of graphons}

\author[Chuangpishit et al.]{Huda Chuangpishit \and Mahya Ghandehari \and Jeannette Janssen}
\address{Department of Mathematics \& Statistics, Dalhousie University, Halifax, Nova Scotia, Canada, B3H 3J5}{
\address{Department of Mathematical Science, University of Delaware, Newark, DE 19716, USA}

\date{\today}

\newcommand{\ignore}[1]{}
\newtheorem{theorem}{Theorem}[section]
\newtheorem{corollary}[theorem]{Corollary}
\newtheorem{case}[theorem]{Exception}
\newtheorem{definition}[theorem]{Definition}
\newtheorem{observation}[theorem]{Observation}

\newtheorem{claim}[theorem]{Claim}
\newtheorem{assumption}[theorem]{Assumption}

\newtheorem{lemma}[theorem]{Lemma}
\newtheorem{proposition}[theorem]{Proposition}

\newcommand{\cal}{\mathcal}
\newcommand{\Rrr}{{\mathbb R}}

\def\sup{{\rm sup}}
\def\range{{\rm range}}

\def\max{{\rm max}}
\def\min{{\rm min}}

\def\P{{ P}}
\def\O{{\mathcal O}}

\def\Q{{ Q}}

\def\Nnn {{\mathbb N}}

\def\dom {{\rm dom}}
\def\range{{\rm range}}

\begin{document}

\begin{abstract} 
Let  $w:[0,1]^2\rightarrow [0,1]$ be a  symmetric function, and consider the random process $G(n,w)$, where vertices are chosen from $[0,1]$ uniformly at random, and $w$ governs the   edge formation probability. Such a random graph is said to have a linear embedding, if the probability of linking to a particular vertex $v$ decreases with distance. The rate of decrease, in general, depends on the particular vertex $v$. A linear embedding is called uniform if the probability of a link between two vertices depends only on the distance between them. 
In this article, we consider the question whether it is possible to ``transform'' a linear embedding to a uniform one,  through replacing the uniform probability space $[0,1]$ with a suitable probability space on ${\mathbb R}$. 
We give necessary and sufficient conditions for the existence of a uniform linear embedding for random graphs where $w$ attains only a finite number of values.
Our findings show that for a general $w$ the answer is negative  in most cases. 
\end{abstract}

\maketitle

\section{Introduction}
In the study of large, real-life networks such as on-line social networks and hyperlinked ``big data" networks, biological networks and neural connection networks, link formation is often modelled as a stochastic process. The underlying assumption of the link formation process is that vertices have an {\sl a priori} identity and relationship to other vertices, which informs the link formation. These identities and relationships can be captured through an embedding of the vertices in a metric space, in such a way that the distance between vertices in the space reflects the similarity or affinity between the identities of the vertices. Link formation is assumed to occur mainly between vertices that have similar identities, and thus are closer together in the metric space.
We take as our point of departure a very general stochastic graph model that fits the broad concept of graphs stochastically derived from a spatial layout, along the same principles as described above. We refer to this model as a {\sl spatial random graph}. In a spatial random graph, vertices are embedded in a metric space, and the {\sl link probability} between two vertices depends on this embedding in such a way that vertices that are close together in the metric space are more likely to be linked. 

The concept of a spatial random graph allows for the possibility that the link probability depends on the spatial position of the vertices, as well as their metric distance. Thus, in the graph we may have tightly linked clusters for two different reasons. On the one hand, such clusters may arise when vertices are situated in a region where the link probability is generally higher. On the other hand, clusters can still arise when the link probability function is {\sl uniform}, in the sense that the probability of a link between two vertices depends only on their distance, and not on their location. In this case, tightly linked clusters can arise if the distribution of vertices in the metric space is inhomogeneous. 
The central question addressed in this paper is how to recognize spatial random graphs with a uniform link probability function.

For our study we ask this question for a general edge-independent random graph model which generalizes the Erd\H{o}s-Renyi random graph $G(n,p)$.  Let ${\cal W}_0$ be the set of symmetric, measurable functions from $[0,1]^2$ to $[0,1]$, and let $w\in {\cal W}_0$. The $w$-random graph $G(n,w)$ is the graph with vertex set $\{ 1,2,\dots ,n\}$ where edges are formed according to a two-step random process. First, each vertex $i$ receives a label $x_i$ drawn uniformly from $[0,1]$. Then, for each pair of vertices $i<j$ independently, an edge $\{ i,j\}$ is added with probability $w(x_i,x_j)$.
Alternatively, the random graph $G(n,w)$ may be seen as a one-dimensional spatial model, where the label $x_i$ represents the coordinate of vertex $i$. In that case, the process $G(n,w)$  can be described as follows: a set $P$ of $n$ points is chosen uniformly from the metric space $[0,1]$. Any two points $x,y\in P$ are then linked with probability $w(x,y)$. For $G(n,w)$ to correspond to the notion of a spatial random graph, $w$ must satisfy a certain type of monotonicity. This is captured by the following definition.

\begin{definition}
\label{def:di}
A function $w\in {\cal W}_0$ is \emph{diagonally increasing} if for every $x, y,z\in [0,1]$, we have
\begin{enumerate}
\item $x\leq y\leq z \Rightarrow w(x,z)\leq w(x,y),$
\item $ y\leq z \leq x \Rightarrow w(x,y)\leq w(x,z)$.
\end{enumerate}
A function $w$ in ${\cal W}_0$ is diagonally increasing almost everywhere if there exists a diagonally increasing function $w'$ which is equal to $w$ almost everywhere.
\end{definition}

Next we formulate our central question: ``which functions $w$ are in fact uniform in disguise?''

\begin{definition}
\label{uniform}
A diagonally increasing function $w\in {\cal W}_0$ has a  {\sl uniform linear embedding} if there exists a measurable injection $\pi:[0,1]\rightarrow \Rrr$ and a decreasing function $f_{pr}:\Rrr^{\geq 0}\rightarrow [0,1]$ such that for every  $x,y\in [0,1]$, $w(x,y)=f_{pr}(|\pi(x)-\pi(y)|)$.
\end{definition}

The function $f_{pr}$ is the {\sl link probability function}, which gives the probability of the occurrence of a link between two points $x,y$ in terms of their distance. The function $\pi$ determines a probability distribution $\mu$  on $\Rrr$, where for all $A\subseteq \Rrr$, $\mu(A)$ equals the Lebesque measure of $\pi^{-1}(A)$.
An equivalent description of $G(n,w)$ is given as follows: a set of $n$ points is chosen from $\Rrr$ according to probability distribution $\mu$. Two points are linked with probability given by $f_{pr}$ of their distance. 
{ Note that in the above definition, we restrict ourselves to injective embeddings $\pi$. We will see in Section 4 that any uniform embedding
must be monotone. Therefore, the requirement that $\pi$ is injective is equivalent to the requirement that the probability distribution $\mu$ has no points of positive measure.}

The notion of diagonally increasing functions, and our interpretation of spatial random graphs, were first given in previous work, see \cite{linearembeddings}. 
In \cite{linearembeddings}, a graph parameter $\Gamma$ is given which aims to measure  the similarity of a graph to an instance of a one-dimensional spatial random graph model. However, the parameter $\Gamma$ fails to distinguish uniform spatial random graph models from the ones which are intrinsically nonuniform. 

In this paper, we give necessary and sufficient conditions for the existence of  uniform linear embeddings for functions in ${\cal W}_0$. We remark that understanding the structure and behaviour of functions in ${\cal W}_0$ is important, due to their deep connection to the study of graph sequences. 
Functions in ${\cal W}_0$ are referred to as {\sl graphons},
and they play a crucial role in the emerging theory of limits of sequences of dense graphs as developed through work of several authors (see \cite{lovaszszegedy2006}, \cite{borgsI2008}, 
\cite{borgsII2007}, \cite{borgs2011}, and also the book \cite{lovaszbook} and the references therein). This theory gives a framework for convergence of sequences of graphs of increasing size that exhibit similar structure.  Structural similarity in this theory is seen in terms of {\sl homomorphism densities}. 
For a given graph $G$, the homomorphism counts are the number of homomorphisms from each finite graph $F$ into $G$, and homomorphism densities are normalized homomorphism counts. The isomorphism type of a (twin-free) graph $G$ is determined by its homomorphism densities. Even partial information about the homomorphism densities of a graph narrows the class of isomorphism types a graph can belong to, and thus gives information about the structure of the graph. 
It was shown in \cite{borgs2011} that the homomorphism densities of the random graph $G(n,w)$ {\sl asymptotically almost surely} approach those of the function $w$. Therefore, $w$ encodes the structure exhibited by  the random graph model $G(n,w)$.

\section{Notation and main result}\label{sec:main}
Recall that $\mathcal{W}_0$ denotes the set of graphons, or all measurable functions  $w: [0, 1]^2\rightarrow [0, 1]$ which are symmetric, {\it i.e.} $w(x,y)=w(y,x)$ for every $x,y\in[0,1]$. 

In this paper, we consider only functions of finite range, for two reasons. Firstly, any function in ${\cal W}_0$ can be approximated by a function with finite range. In fact, for any $\alpha \in [0,1]$, we can create a finite valued $w^{\alpha}$ by rounding all values of $w$ to the nearest multiple of $\alpha$. In order for $w$ to have a uniform embedding, $w^{\alpha}$ has to have a uniform embedding, for each $\alpha$. Secondly, we will see that the necessary conditions become increasingly restrictive when the size of the range increases. This leads us to believe that for infinite valued functions, either the uniform linear embedding will be immediately obvious when considering $w$, or it does not exist. 

Throughout this paper, unless otherwise stated, we assume that $w\in{\cal W}_0$ is a diagonally increasing function with finite range.
It turns out that $w$ has a very specific form. For each $y\in [0,1]$, $w(x,y)$, viewed as a function of $x$, is a step function which is increasing for $x\in [0,y]$, and decreasing for $x\in [y,1]$. This function is determined by the boundary points where the function changes values. This leads to the following definition. 

\begin{definition}
\label{def:boundaries}
Let $w\in{\cal W}_0$ be a diagonally increasing function with  $\range(w) = \{\alpha_1,\ldots,\alpha_N\}$, where $\alpha_1> \alpha_2>\ldots>\alpha_N$.
For $1\leq i\leq  N$, the {\sl upper boundary} $r_i$ and the {\sl lower boundary } $\ell_i$ are functions from $[0,1]$ to $[0,1]$ defined as follows. Fix $x\in[0,1]$. Then
\[\ell_i(x)=\inf\{y\in[0,1]:\, w(x,y)\geq \alpha_i\},\] 
and
\[r_i(x)=\sup\{y\in[0,1]:\, w(x,y)\geq \alpha_i\}.\] 
Also, for $1\leq i< N$, define $r^*_i=r_i|_{[0,\ell_i(1)]}$ and $\ell^*_i=\ell_i|_{[r_i(0),1]}$. Note that $r_i^*$ has domain $[0,\ell_i(1)]$ and range $[r_i(0),1]$, and $\ell_i^*$ has domain $[r_i(0),1]$ and range $[0,\ell_i(1)]$.
 \end{definition}

\begin{figure}[ht]
\centerline{\includegraphics[width=0.45\textwidth]{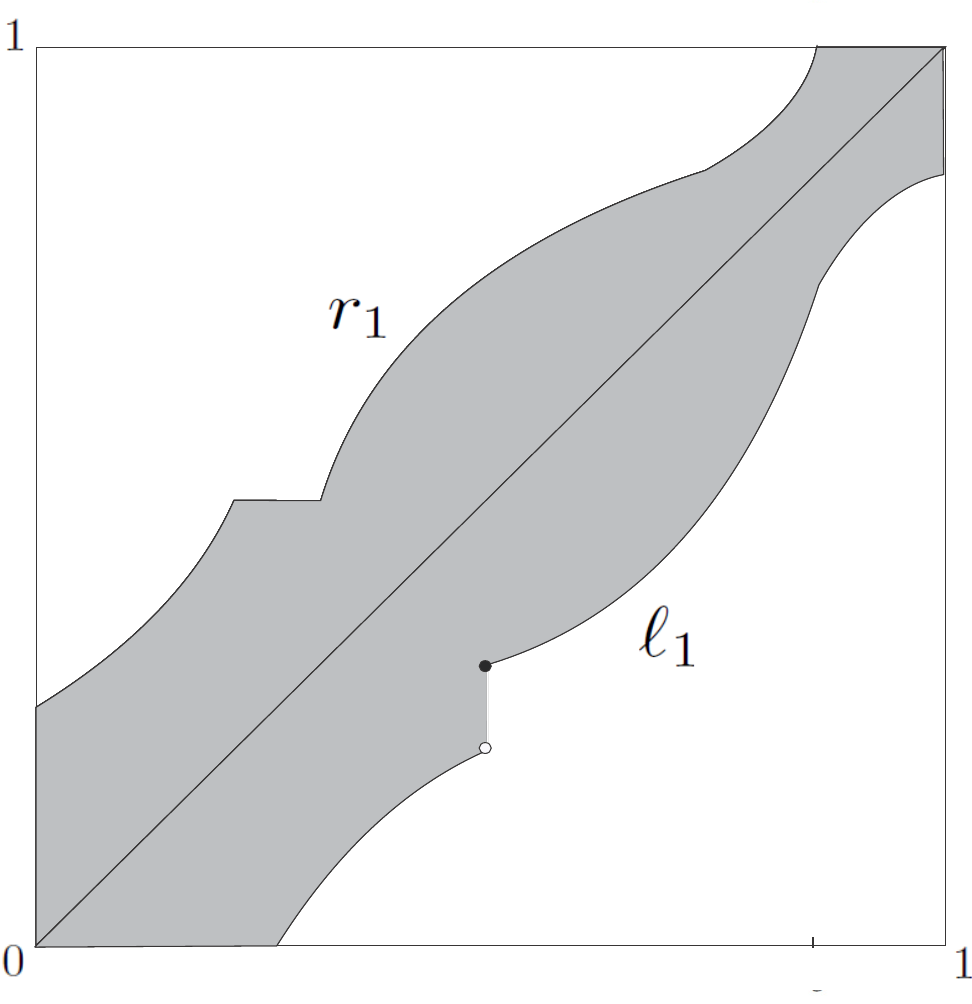}}
\caption{An example of a two-valued, diagonally increasing graphon $w$. The grey area is where $w$ equals $\alpha_1$, elsewhere $w$ equals $\alpha_2$. The functions $\ell_1$ and $r_1$ form the boundaries of the grey area.\label{fig:w}}
\end{figure}

Since $w$ is diagonally increasing, we have $w(x,y)\geq \alpha_i$ if  $y\in(\ell_i(x),r_i(x))$. On the other hand, $w(x,y)<\alpha_i$ whenever $y\in[0,\ell_i(x))\cup(r_i(x),1]$. Thus, the boundaries almost completely define $w$. By modification of $w$ on a set of measure zero, we can assume that $w(x,r_i(x))=w(x,\ell_i(x))=\alpha_i$.  
We will assume throughout without loss of generality that for $1\leq i\leq N$, the functions $\ell_i,r_i:[0,1]\rightarrow [0,1]$  satisfy
\begin{equation}\label{eqn:finite}
w(x,y)\geq \alpha_i \mbox{ if and only if } \ell_i(x)\leq y\leq r_i(x).
\end{equation}

Note also that, for all $x$, $r_{N}(x)=1$ and $\ell_N(x)=0$. Therefore, we usually only consider the boundary functions $r_i,\ell_i$ for $1\leq i<N$. 
Below, we state an adaptation of Definition \ref{uniform} for finite-valued diagonally increasing functions.

\begin{definition}
\label{def:uniformfinite}
Let $w\in {\cal W}_0$ be a diagonally increasing function of finite range, defined as in Equation (\ref{eqn:finite}). Then $w$ has 
 a uniform linear embedding  if there exists a measurable injection $\pi:[0,1]\rightarrow \Rrr$ and 
real numbers $0<d_1<d_2<\dots <d_{N-1}$ so that for all $(x,y)\in [0,1]^2$,
 \begin{equation}\label{eq:finiterange}
w(x,y)=\left\{
\begin{array}{cl}
  \alpha_1 & \text{if }|\pi(x)-\pi(y)|\leq d_1 , \\
  \alpha_i &\text{if } d_{i-1}< |\pi(x)-\pi(y)|\leq d_i \text{ and }1<i<N,\\
  \alpha_N &\text{if } |\pi(x)-\pi(y)|> d_{N-1}.
\end{array}
\right.
\end{equation}
We call $d_1, d_2,\dots, d_{N-1}$ the parameters of the uniform linear embedding $\pi$.
\end{definition}

\begin{remark} 
We do not restrict ourselves to linear uniform embeddings that are continuous. Namely, such a restriction would go against the grain of the question we are trying to answer. Basically, if a uniform linear embedding exists, then the random graph has a natural geometric interpretation: vertices are embedded in a linear space, and the probability of an edge between vertices depends only on the distance between the vertices. Discontinuity in the embedding corresponds to ``empty'' space on the line. If the random graph is a model for real-life networks, then such empty spaces may indicate real obstacles to the appearance of vertices in certain regions. Excluding this option limits the possibilities. As we will see, this limitation is unnecessary.

\end{remark}

We say that a function $w$ is \emph{well-separated} if the boundaries $r^*_i$ and $\ell^*_i$ are continuous, and have positive distances from the diagonal, and from each other.  Formally, well-separated boundaries have the property that there exists $\epsilon >0$ so that, for all $i$ and for all $x\in [0,\ell_i(1)]$, $r_i(x)-x\geq\epsilon$, and $r_j(x)-r_i(x)\geq \epsilon$ for $j>i$. In particular, this implies that $r_1(0)>0$ and $\ell_1(1)<1$. { It also implies that $w$ assumes its maximum value on the entire diagonal. } In the remainder of the article, we will assume that $w$ is well-separated. 

The domain and range of $\ell_i^*$ and $r_i^*$, and thus the domain and range of any composition of such functions, are (possibly empty) closed intervals.
We will refer to  $f_1\circ\ldots\circ f_k$ as a \emph{legal composition}, if each $f_i$ belongs to $\{r^*_j,\ell^*_j:\ j=1,\ldots, N-1\}$ and $\dom(f_1\circ\ldots\circ f_k)\neq \emptyset$.
We define the \emph{signature} of the legal composition $f_1\circ\ldots\circ f_k$ to be the $(N-1)$-tuple $(m_1,\ldots,m_{N-1})$, where $m_i$ is the number of occurrences of $r_i^*$ minus the number of occurrences of $\ell_i^*$ therein. We use Greek letters such as $\phi,\psi,\dots$ to denote legal compositions. We emphasize that a legal composition is a function, which we denote by legal function, presented in a particular manner as a composition of boundary functions.  
Note that two legal compositions may be identical as functions, but have different signatures, due to difference in their presentations. 
Legal compositions provide us with appropriate ``steps'' to define constrained points.

\begin{definition}\label{def:PQ}
Let $w\in{\cal W}_0$ be a diagonally increasing function with finite range. Keep notations as in Definition \ref{def:boundaries}, and define
\begin{eqnarray*}
P&=&\{\phi(0): \ \phi \mbox{ is a legal composition with } 0\in\dom(\phi)\},\\
Q&=&\{\psi(1): \ \psi \mbox{ is a legal composition with } 1\in\dom(\psi)\}.
\end{eqnarray*}
We refer to $P\cup Q$ as the set of {\sl constrained points} of $w$. 
\end{definition}
We will see later in the paper that either $P$ and $Q$ are disjoint, or $P=Q$. In the case $P=Q$, a uniform linear embedding is more constrained, but generally the conditions for the existence of a uniform linear embedding are identical to the (easier) case where $P$ and $Q$ are disjoint. { However, for one special case, our techniques do not suffice for the establishment of necessary conditions. In particular, for this case it may not hold that any uniform embedding is 
monotone. }
This is the exceptional case where $P=Q$, and  the boundary functions are too far apart from the diagonal. { Since $P=Q$ only under special circumstances (see Section 4.1), this exception does not 
detract substantially from our main result.}
We will exclude this case from our main theorem, and from our discussions. 
\begin{case}\label{SpecialCase}
Suppose that $P=Q$. If $r_1(0)>\ell_{N-1}(1)$ and $r_{N-1}(0)> \ell_1(1)$, then the conditions given in Theorem \ref{thm:nec-suff-cond-N} are sufficient, but may not be necessary.
\end{case}
The choice of terminology ``legal composition'' and ``constrained points'' becomes apparent through our results.  The next proposition, which applies to the special case where $\pi$ is continuous, makes this clear. We will see later, in Subsection \ref{subset:displacement-necessary}, that Proposition
\ref{prop:cts-case} has analogies in the general case where $\pi$ is not necessarily continuous. These results are in fact the foundations of the approach taken in this article. 
\begin{proposition}\label{prop:cts-case}
Let $w\in {\cal W}_0$ be a diagonally increasing function with  finite range. Suppose $\pi:[0,1]\rightarrow {\mathbb R}$ is a continuous uniform linear embedding with parameters 
$0<d_1<d_2<\ldots<d_{N-1}$ as in Definition \ref{def:uniformfinite}. Then, for every $1\leq i\leq N-1$, we have
$$|\pi(r_i^*(x))-\pi(x)|=d_i\ \mbox{ and }\ |\pi(\ell_i^*(x))-\pi(x)|=d_i,$$ 
whenever $x$ is in the appropriate domain.
\end{proposition}
\begin{proof}
Let $1\leq i\leq N-1$, and assume that $x\in \dom(r_i^*)$. Clearly $w(x,r_i^*(x))=\alpha_i$, and $w(x,z)<\alpha_i$ whenever $z>r_i^*(x)$. Since $\pi$ is a uniform linear embedding, we 
have $|\pi(x)-\pi(r_i^*(x))|\leq d_i$ and $|\pi(x)-\pi(z)|>d_i$ for every $z>r_i^*(x)$. These inequalities, together with the continuity of $\pi$, imply that $|\pi(x)-\pi(r_i^*(x))|=d_i$. The second statement follows analogously. 
\end{proof}
Let $\phi$ be a legal composition, and $p=\phi(0)$ be a point in $\P$. We will see later, in Proposition \ref{prop:pi-properties}, that repeated application of Proposition \ref{prop:cts-case} implies that the image $\pi(p)$ of $p$ is determined by the signature of $\phi$. This inspires the following definition.
\begin{definition} 
\label{def:displacement} 
Assume a positive integer $N$ and real numbers $d_{N-1}>\ldots >d_1>0$ are given. The {\em displacement} of a legal composition $\phi$, denoted by $\delta (\phi)$ is defined as 
\[
\delta (\phi)=d_1m_1+\ldots +d_{N-1}m_{N-1},
\]
where $(m_1,m_2,\ldots ,m_{N-1})$ is the signature of $\phi$.
\end{definition}

Now, we can state our main theorem.
\begin{theorem}[Necessary and sufficient conditions]\label{thm:nec-suff-cond-N}
Let $w$ be a well-separated finite-valued diagonally increasing function assuming values $\alpha_1>\ldots>\alpha_N$.  Let $\P$ and $\Q$  be as in Definition \ref{def:PQ}, and exclude \ref{SpecialCase}.  
The function $w$ has a uniform linear embedding if and only if  the following conditions hold:
\begin{enumerate}
\item If $\phi$ is a legal function with $\phi(x)=x$ for some $x\in \dom(\phi)$, then $\phi$ is the identity function on its domain. 

\item  There exist real numbers $0<d_1<\ldots<d_{N-1}$ such that 
\begin{itemize}
\item[(2a)] The displacement $\delta$ as defined in Definition \ref{def:displacement} is increasing on $\P$, in the sense that, for all $x,y\in\P$, and legal compositions $\phi,\psi$ so that $x=\phi(0)$ and $y=\psi(0)$, we have that, if  $x<y$ then $\delta(\phi)<\delta(\psi)$.
\item[(2b)] If $\P\cap\Q=\emptyset$ then there exists $a\in \mathbb{R}^{\geq 0}$ which satisfies the following condition:
If $\phi$ and $\psi$ are legal compositions with $1\in\dom(\phi)$ and $0\in \dom(\psi)$, and  if $\phi(1)<r_i^*(0)$  then 
$\delta(\psi)<a<d_i-\delta(\phi)$. 
\end{itemize}
\end{enumerate}
\end{theorem}

Let us briefly discuss the necessary and sufficient conditions given in the above theorem 
from the algorithmic perspective. 
The sets $P$ and $Q$ can be generated in stages, starting with  $\P_0=\{ 0\}$ and $\Q_0=\{ 1\}$, and iteratively constructing $\P_n$ and $Q_n$ by applying the functions $r^*_i$ and $\ell_i^*$ to the points in $\P_{n-1}$ and $\Q_{n-1}$. Conditions (2a) and (2b) translate into simple inequalities on the sets of constrained points, and can be checked during the generation process. (It can be shown that the time needed is bounded by a polynomial in the number of constrained points.) Thus,
when a violation is found, this is proof that no uniform embedding exists, and the process can be stopped. 
 
Our proof of sufficiency is constructive. When the set of constrained points is finite { and the conditions are satisfied}, the construction { of a uniform embedding} presented in the proof of Theorem \ref{thm:nec-suff-cond-N} can be implemented with an algorithm with complexity polynomial in the size of the set of constrained points. (See Section \ref{sec:sufficient}). 
Condition (1) 
involves checking the equivalence of a function on the real interval $[0,1]$ to the identity function, which cannot generally be achieved in finite time.  If this condition is violated but no violation of condition (2) is found, then the uniform embedding as constructed in Section \ref{sec:sufficient} may not be well-defined. 

The construction of $\pi$ given in Section \ref{sec:sufficient} { and the iterative generation of the sets of constrained points} suggest a possible approximation algorithm. 
One could stop the generation process after a finite number of steps, and construct an approximate embedding $\pi$ based on the set  $\P_n$ containing only a pre-described number $n$ of generations. It is natural to believe that this approximate embedding is ``close'' to $w$, and will converge to $w$ when more generations are taken into account. We will  investigate this approach in further work.


\section{Properties of legal compositions}\label{sec:conditions}
In this section, we give some general properties of boundary functions and the legal compositions derived from the boundaries. Throughout the rest of this article, we 
assume $w$ is as given in the following assumption.

\begin{assumption}\label{ass:w}
Let $w\in{\cal W}_0$ be a diagonally increasing function assuming values $\alpha_1>\dots >\alpha_N$. Let the boundaries of $w$ be as in Definition \ref{def:boundaries}. 
Assume that $w(x,y)\geq \alpha_i$ if and only if $\ell_i(x)\leq y\leq r_i(x)$.
Moreover, assume that $w$ is well-separated.
\end{assumption}

We start by giving an example, which illustrates why legal compositions play such an important role. 

\subsubsection*{Example 1}\label{example-negative}
Let $w$ be a diagonally increasing $\{\alpha_1,\alpha_2,\alpha_3\}$-valued function where $\alpha_1>\alpha_2>\alpha_3=0$ and $w$ has the following upper boundaries. 
$$
r_1(x)=\left\{
\begin{array}{cc}
x+\frac{1}{10} & x\in[0,\frac{9}{10}]  \\
1  &  x\in [\frac{9}{10},1]\\
\end{array}
\right.,$$
and 
$$
r_2(x)=\left\{
\begin{array}{cc}
2x+\frac{1}{8} & x\in[0,\frac{7}{16}]  \\
1  &  x\in [\frac{7}{16},1]\\
\end{array}
\right..$$
Define sequences $\{x_i\}_{i\geq 0}$ and $\{y_j\}_{j\geq 0}$ as follows. Let $x_0=0$ and $x_i=r_1^i(0)=\frac{i}{10}$ for $1\leq i\leq 9$. Also let $y_0=0$ and $y_j=r_2^j(0)=\frac{2^j-1}{8}$ for $1\leq j \leq 3$. Suppose that $w$ admits a uniform embedding $\pi$ with parameters $d_2>d_1>0$. We will see later ({\sl e.g.} Corollary \ref{cor:increasing}) that $\pi$ can be assumed to be strictly increasing and $\pi(0)=0$.
The image under $\pi$ of all points $x_i$ and $y_i$ are almost completely determined. Precisely, we must have that
\begin{eqnarray}\label{eq:exp1}
(i-1)d_1<\pi(x_i)\leq id_1\quad \mbox{and}\quad (j-1)d_2<\pi(y_j)\leq jd_2
\end{eqnarray}
Equation \ref{eq:exp1}, together with the fact that $x_8<y_3$, implies that $\pi(x_2)\leq 2d_1$ and $7d_1<\pi(x_8)<\pi(y_3)\leq 3d_2$. Also, since $x_2>y_1=r_2(0)$, we have that $\pi(x_2)>d_2$.
This gives contradicting restrictions on $d_1$ and $d_2$, and thus the function $w$ has no uniform linear embedding. 

The reason for non-existence of a uniform linear embedding here, in contrast to the fact that the boundary functions behave nicely, is the uneven distribution of points obtained from repeated application of $r_1$ and $r_2$ to $0$. More precisely in the interval $[y_1,y_2]$ there are two points of the sequence $\{x_i\}_{i\geq 0}$ while in the interval $[y_2,y_3]$ there are five of them. This does not allow the function $w$ to have a uniform linear embedding. 
\vspace*{0.2cm}

Note that the points $\{x_i\}_{i=1}^9$ and $\{y_i\}_{i=1}^3$ of the above example are in the set $\P$ as defined in Definition \ref{def:PQ}. This suggests that the points in $\P$ play an important role in determining whether a uniform embedding exists. Conditions (2a) and (2b) of Theorem \ref{thm:nec-suff-cond-N} indicate, roughly, that the points in $\P$ and $\Q$ do not give rise to a contradiction as described here.

Let us therefore take a closer look at boundaries and legal compositions, which define the points in $\P$ and $\Q$. Let $w$ be as in Assumption \ref{ass:w}. We first note that the boundaries  are increasing. Namely, assume to the contrary that there exist $x,y$ so that $x<y$ and $\ell_i (x) >\ell_i (y)$. Then there exists $z$ so that 
$\ell_i (y) < z <\ell_i(x) \leq x <y$. So $z\in [\ell_i(y),r_i(y)]$ and $z\not\in [\ell_i(x),r_i(x)]$, and thus $w(x,z)<w(y,z)$. This contradicts the fact that $w$ is diagonally increasing. 

The lower and upper boundaries of $w$ are closely related, as $w$ is symmetric. 
Firstly,  a discontinuity in $\ell_i$ corresponds to an interval where $r_i$ is constant. Thus if $r_i^*=r_i|_{[0,\ell_i(1)]}$ and $\ell_i^*=\ell_i|_{[r_i(0),1]}$ are both strictly increasing, they are both continuous. Recall that these functions completely determine $w$.
%
Next note that  for all $y<x$ we have that $y\geq\ell_i (x)$ if and only if $x\leq r_i(y)$, since $w$ is symmetric. Thus, the upper boundaries are enough to completely determine $w$.   In fact, if $r_i$ is  continuous (thus strictly increasing) on $[0,\ell_i(1)]$,  then $r_i$ can be realized as $\ell_i^{-1}$ on $[0,\ell_i(1)]$, and 1 everywhere else. 
So if $w$ is well-separated, then $r^*_i$ and $\ell_i^*$ are bijective (hence strictly increasing and invertible) functions. Moreover, $r^*_i\circ \ell^*_i$ and $\ell^*_i\circ r^*_i$ are identity functions on their domain.

\begin{observation}\label{obs1}
From the above it is
clear that every legal composition $f_1\circ\ldots\circ f_k$ is a strictly increasing function. We also remark that each term $f_i$ of a legal composition is invertible in the sense that  if $(f_1\circ\ldots\circ f_k)(z)=x$  then $(f_2\circ\ldots\circ f_k)(z)= f_1^{-1}(x)$. Moreover, $f_i^{-1}\in\{\ell_i^*,r_i^*:\, 1\leq i< N\}$.   
\end{observation}

Recall that a legal composition is called a legal function, when considered only as a function. We will now show that the points in $\P$ and $\Q$ are closely related, and that the domain and range of legal functions  are intervals which begin at a point in $\P$, and end at a point in $\Q$. 

\begin{lemma}\label{lemma:domains-of-legal}
Let $w$ be as in Assumption \ref{ass:w} and, let $\phi=f_1\circ\ldots\circ f_k$ be a legal function. 
Then 
\begin{itemize}
\item[(i)] Suppose $z\in \dom(\phi)$. Then $z\in \P$ if and only if $\phi(z)\in \P$. Similarly, $z\in \Q$ if and only if $\phi(z)\in \Q$.
\item[(ii)] There is a one-to-one correspondence between $\dom(\phi)\cap \P$and $\range(\phi)\cap \P$, and between   $\dom(\phi)\cap \Q$ and $\range(\phi)\cap \Q$.
\end{itemize}
\end{lemma}
\begin{proof}
Let $\phi$ be a legal function, and $z\in \dom(\phi)$.  By definition, if $z\in \P$ then there exists a legal function $\psi$ such that $z=\psi(0)$. Thus, $\phi(z)=(\phi\circ \psi)(0)$ belongs to $\P$ as well. On the other hand, assume that $\phi(z)\in \P$, {\it i.e.} $\phi(z)=\eta(0)$ for a legal function $\eta$. By Observation \ref{obs1}, we have $(\phi^{-1}\circ \eta)(0)=z $, thus $z\in \P$.
A similar argument proves the statement for $z\in\Q$. Part (ii) trivially follows from (i).
\end{proof}
Note that for every legal composition $\phi=f_1\circ\ldots\circ f_k$ applied to a point $x$ there is a sequence of points which lead from $x$ to $\phi(x)$, by first applying $f_k$, then $f_{k-1}$ etc. This inspires the definition of ``orbit'' of an element in the domain of a legal composition. 
\begin{definition}
Under the Assumption \ref{ass:w}, let $\phi=f_1\circ\ldots\circ f_k$ be a legal composition, and  $x\in \dom(\phi)$. The \emph{orbit} of 
$x$ under $\phi$, denoted by $\O_x$, is defined to be the set $\{x\}\cup\{f_t\circ\ldots\circ f_k(x): 1\leq t\leq k\}$. We sometimes use the notation 
$$\O_x=\{x_0,x_1,\ldots,x_k\},$$
where $x_0=x$ and $x_i=f_{k-i+1}\circ \dots \circ f_k(x)$ for $1\leq i\leq k$. 

We say that $\phi(x)$ \emph{touches} 0 (or 1) if the orbit of $x$ under $\phi$ includes 0 (or 1). 
\end{definition}

\begin{proposition}\label{prop:orbits-domain}
Under Assumption \ref{ass:w}, if  $\phi=f_1\circ\ldots\circ f_k$ is a legal composition with $\dom(\phi)=[p,q]$ then we have:
\begin{itemize}
\item[(i)] $p\in\P$ and $q\in\Q$. 
\item[(ii)] $\phi(p)$ touches 0, and $\phi(q)$ touches 1. 
%
\end{itemize}
\end{proposition}
\begin{proof}
We use induction on the number of terms of the legal composition. Clearly (i) holds for all functions  $\phi\in \{r_i^*,\ell_i^*:\, 1\leq i< N\}$.  Now let 
$\phi$ and $\psi$ be legal functions with $\dom(\phi)=[p,q]$,  $\dom(\psi)=[p',q']$, $p,p'\in \P$, and $q,q'\in \Q$. (Recall that by definition of a legal composition $\dom(\phi)$ and $\dom(\psi)$ 
are nonempty.) Since $\phi$ and $\psi$ are strictly increasing, we have 
$\range(\phi)=[\phi(p),\phi(q)]$  and  $\range(\psi)=[\psi(p'),\psi(q')]$. By  Lemma \ref{lemma:domains-of-legal}, we have $\phi(p), \psi(p')\in \P$, and $\phi(q),\psi(q')\in \Q$. 
We know that $\dom(\phi\circ \psi)=\dom(\psi)\cap \psi^{-1}(\dom(\phi)\cap \range(\psi))$. Since $\range(\psi)$, $\dom(\phi)$, and
$\dom(\psi)$ are closed intervals with left bounds in $\P$ and right bounds in $\Q$, we conclude that $\dom(\phi\circ \psi)$ is a closed interval of the same type, {\it i.e.} $\dom(\phi\circ \psi)=[p_1,q_1]$ with $p_1\in\P$ and $q_1\in \Q$.   This proves (i).

We use induction again to prove the statement regarding $p$ in (ii). The proof for $q$ is similar. Assume that $0<p$, since we are done otherwise. First note that (ii) holds for all functions  $\phi\in \{r_i^*,\ell_i^*:\, 1\leq i< N\}$. 
Now consider a legal function $\phi=f_1\circ\ldots\circ f_k$ with domain $[p,q]$. Let $\psi=f_2\circ\ldots\circ f_k$, and assume that $\dom(\psi)=[p',q']$. 
First observe that $p\geq p'$ as $[p,q]\subseteq [p',q']$. 
By induction hypothesis, we know that $\psi(p')$ touches zero. If $p=p'$, then $\phi(p)$ touches zero as well. So suppose that $p> p'$. This means that $\psi(p')\not\in\dom(f_1)$ but $\psi(p)\in\dom(f_1)$. 
So $f_1=\ell_j^*$ for some $j$, because $\psi(p')<\psi(p)$.
 Moreover, we must have $p=\psi^{-1}(r_j^*(0))$. But this implies that 
$$\phi(p)=\ell_j^*\circ\psi(p)=\ell_j^*\circ\psi(\psi^{-1}(r_j^*(0)))=0,$$
which finishes the proof.
\end{proof}

We end this section with a look at the special case where $N=2$, so there is only one pair of boundary functions, which we will call $r$ and $\ell$. In this case, $\P$ and $\Q$ are very simple, and the conditions of Theorem \ref{thm:nec-suff-cond-N} are trivially satisfied. Thus, a uniform linear embedding always exists. We state and prove this formally in the following proposition.

\begin{proposition}\label{example-positive}
Let $w$ be a well-separated two-valued diagonally increasing function with  upper and lower boundaries $r$, $\ell$ respectively. Then there exists a uniform linear embedding of $w$.
\end{proposition}

\begin{proof}
Suppose that $w(x,y)=\alpha_1$ if $\ell(x)\leq y\leq r(x)$, and $w(x,y)=\alpha_2$ otherwise. 
Let $x_0=0$ and $x_i=r^i(0)$ for $i\geq 1$. Since $w$ is well-separated, $r$ has positive distance from the diagonal, and thus there exists $k\in \mathbb{N}$ such that $r^k(0)<1$ and $r^{k+1}(0)=1$. Note that $\{x_i\}_{i=0}^k$ is a strictly increasing sequence, as $w$ is well-separated. Define the function $\pi:[0,1]\rightarrow \Rrr^{\geq 0}$ as follows. 
\begin{equation*}
\pi(x)= 
\left\{
\begin{array}{ll}
{\frac{x}{x_1}}  & \mbox{if }x\in [x_0,x_1],\\
\pi(\ell^i(x))+i & \mbox{if }x\in (x_i,x_{i+1}]\mbox{ for }1\leq i\leq k-1,\\
\pi(\ell^k(x))+k & \mbox{if }x\in (x_k,1].
\end{array}
\right.
\end{equation*} 
The function $\pi$ is well-defined and strictly increasing. We now prove that $\pi$ is a uniform linear embedding of $w$, namely
\begin{eqnarray}\label{eq:exp}
w(x,y)=\left\{
\begin{array}{cc}
  \alpha_1 & |\pi(x)-\pi(y)|\leq 1  \\
  \alpha_2 & |\pi(x)-\pi(y)|> 1   
\end{array}
\right..
\end{eqnarray}
{To do so, partition $[0,1]$ using the intervals $J_0=[x_0,x_1], J_1=(x_1,x_2],\ldots,J_k=(x_k,1]$. Let $x,y\in [0,1]$, and $x<y$. The inequality $|\pi(x)-\pi(y)|\leq 1$ holds precisely when either (i)\space $x,y$ belong to the same interval, say $J_i$, or (ii)\space they belong to consecutive intervals, say $J_i$ and $J_{i+1}$, and $x\geq \ell(y)$. If case (ii) happens, clearly $w(x,y)=\alpha_1$, as $\ell(y)\leq x<y$.
In case (i), $r^i(0)\leq x<y\leq r^{i+1}(0)$. Since $r$ is increasing, $r^{i+1}(0)\leq r(x)$, which implies that $x<y\leq r(x)$. Thus $w(x,y)=\alpha_1$ in this case as well. A similar argument proves that if $|\pi(x)-\pi(y)|> 1$ then $w(x,y)=\alpha_2$.}
\end{proof}
\begin{remark}{
Theorem \ref{example-positive} does not hold if we remove the condition that $w$ should be well-separated.
To show this, and to justify our conditions for $w$ in Theorem \ref{thm:nec-suff-cond-N}, let us consider the following simple example.
Suppose there exist $z_1< z_2$ in $(0,1)$ such that $r(z_1)=\ell(z_1)=z_1$, $r(z_2)=\ell(z_2)=z_2$ and for every $x\in(0,1)$, $\ell(x)<r(x)$.
\begin{figure}[h]
\centerline{\includegraphics[width=0.35\textwidth]{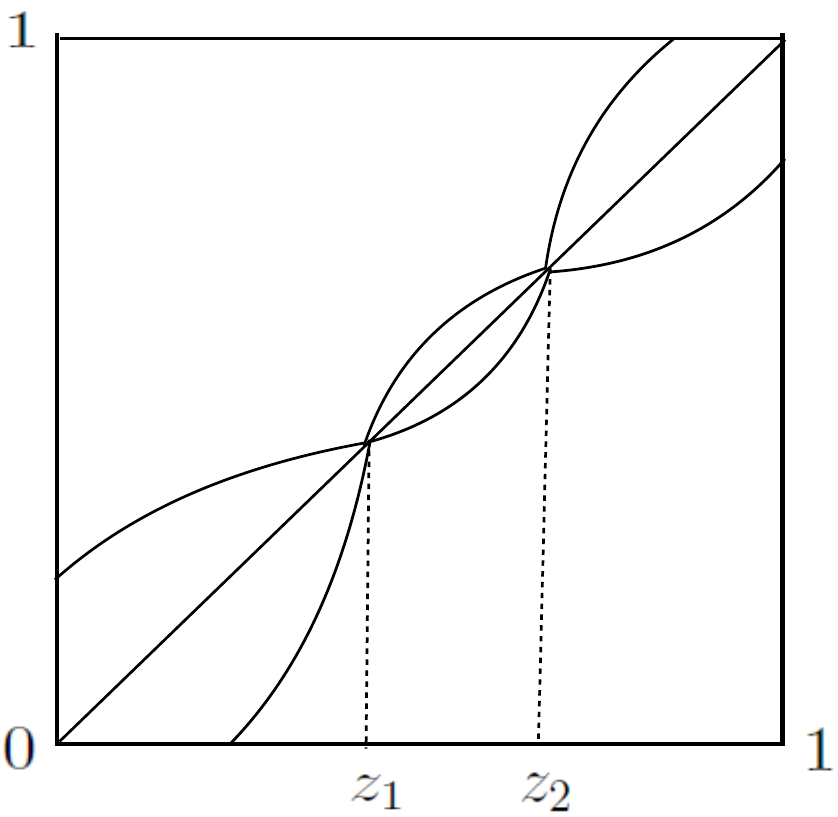}}
\end{figure}
Fix $x_0\in (z_1,z_2)$ and $x'_0\in(z_2,1)$, and  note that 
$$z_1=\ell(z_1)< \ell(x_0),\ \ r(x_0)< r(z_2)=z_2\ \mbox{ and }\ z_2=\ell(z_2)<\ell(x'_0).$$ 
Therefore the sequences $\{x_i\}_{i\in{\mathbb N}}$, $\{y_i\}_{i\in{\mathbb N}}$ and $\{x'_i\}_{i\in{\mathbb N}}$ are all infinite sequences, where $x_i=r^i(x_0)$, $y_i=\ell^i(x_0)$ and $x'_i=\ell_i(x'_0)$  for every $i\in{\mathbb N}$. Moreover, $x_i,y_i\in (z_1,z_2)$ and $ x'_i\in (z_2,1)$. 
If a uniform linear embedding $\pi$ exists, then for every $i\in{\mathbb N}$, we have
$$|\pi(x_{i})-\pi(x_{i+1})|\leq 1, \  \ |\pi(x_i)-\pi(x_{i+2})|>1, \  \ |\pi(x_i)-\pi(z_2)|>1.$$
Therefore, all of the points of the sequence $\{\pi(x_i)\}_{i\in{\mathbb N}}$ lie on one side of $\pi(z_2)$, say  inside the interval $(\pi(z_2)+1,\infty)$, in such a way that consecutive points have distance less than one. Moreover, $\pi(x_i)\rightarrow \infty$. Similar arguments, together with inequality $|\pi(x_i)-\pi(y_j)|>1$, imply that the sequence $\{\pi(y_i)\}_{i\in{\mathbb N}}$ is distributed in a similar fashion in $(-\infty, \pi(z_1)-1)$ and $\pi(y_i)\rightarrow -\infty$. Clearly, this process cannot be repeated for the sequence $\{\pi(x'_i)\}_{i\in{\mathbb N}}$. Hence a uniform linear embedding does not exist. }
\end{remark}

\section{Necessary properties of a uniform linear embedding}
\label{sec:prelim}

We devote this section to the study of uniform linear embeddings, and their interplay with boundary functions. We will assume that a uniform linear embedding $\pi$ exists, and show how the boundaries and points in $\P$ severely restrict $\pi$.
We show that, in essentially all cases, a uniform linear embedding is necessarily strictly monotone. 
Since a uniform linear embedding $\pi$ is not assumed to be continuous in general, we define the left limit $\pi^-$ and the right limit $\pi^+$ at each point, and study their behaviour with respect to the boundaries. Finally, we will use these results to prove the necessity part of Theorem \ref{thm:nec-suff-cond-N}.

Comparing Assumption \ref{ass:w} and Definition \ref{def:uniformfinite} we see that  $\pi$ is a uniform linear embedding if and only if
for all $x\in \dom(r_i^*)$ and $1\leq i< N$, 
\begin{equation}\label{eq:link-r}
\begin{array}{ll}
|\pi(z)-\pi(x)|\leq d_i&\text{if } x\leq z\leq r_i^*(x)\mbox{ or } x\not\in\dom(r_i^*)\\
|\pi(z)-\pi(x)|>d_i &\text{if } z> r_i^*(x).
\end{array}
\end{equation}
Similarly,  for all $x\in \dom(\ell_i^*)$ and $1\leq i< N$, 
\begin{equation}\label{eq:link-l} 
\begin{array}{ll}
|\pi(z)-\pi(x)|\leq d_i&\text{if }   \ell_i^*(x)\leq z\leq x \mbox{ or } x\not\in\dom(\ell_i^*)\\
|\pi(z)-\pi(x)|>d_i &\text{if } z<\ell_i^*(x).
\end{array}
\end{equation}
These conditions simplify the question of existence of uniform linear embeddings.

Recall that, in this section we assume that a linear uniform embedding exists. We now state this as a formal assumption.

\begin{assumption}\label{ass:pi-1}
Let $w$ be as in Assumption \ref{ass:w}. Assume that
we are not in the case of Exception \ref{SpecialCase}. Assume that $w$ admits a uniform linear embedding $\pi:[0,1]\rightarrow\Rrr$ with parameters $0<d_1<d_2<\dots <d_{N-1}$ as in Definition \ref{def:uniformfinite}. Assume without loss of generality that $\pi(0)<\pi(r_1^*(0))$, and $\pi(0)=0$.
\end{assumption}

From (\ref{eq:link-r}) and (\ref{eq:link-l}), we obtain important properties  for a uniform linear embedding $\pi$. These are listed in Corollaries \ref{cor:increasing} and Proposition \ref{prop:pi-properties} below. {Indeed in what follows, we prove that if a uniform linear embedding $\pi$ exists then it must be monotone. The monotonicity of $\pi$ for the cases $\ell_1(1)>r_1(0)$ and $\ell_1(1)<r_1(0)$ are proved using different techniques. The reason is that 
when $\ell_1(1)>r_1(0)$, there are at least two points of the form ${r_i^*}^k(x)$. However, the condition $\ell_1(1)<r_1(0)$ implies that the maximum integer $k$ for which ${r_i^*}^k(x)$ is defined is 1, for all $1\leq i\leq N-1$.
If $\ell_1(1)= r_1(0)$, then $\P=\Q$ and $r_{N-1}(0)>r_1(0)= \ell_1(1)$, a case that was excluded by Exception \ref{SpecialCase}.}
\begin{lemma}\label{lemma:chain}
Let $w$ and $\pi$ be as in Assumption \ref{ass:pi-1}. Then we have the following. Let  $x\in [0,1]$.  Let $k_\max$ be the largest positive integer with ${r_1^*}^{k_\max}(x)<1$. Then the sequence $\{\pi({r_1^*}^k(x))\}_{k=0}^{k_\max}$ is strictly increasing. Moreover $\pi({r_1^*}^{k_\max}(x))<\pi(1)$.
\end{lemma}

\begin{proof}
For $0\leq k\leq k_\max$, let $x_k={r_1^*}^{k}(x)$, and let $x_{k_\max +1}=1$. First we 
show that the sequence $\{ x_i\}_{i=1}^{k_\max+1}$ is strictly monotone, and hence, by our assumption, strictly increasing.

By Assumption \ref{ass:pi-1}, $\pi(x_1)>\pi(x_0)$. 
We will show by induction that for each $k$, $1\leq k\leq k_\max +1$, $\pi(x_k)>\pi(x_{k-1})$. The base case is assumed. Let $k>1$, and assume that $\pi(x_{k-2})<\pi(x_{k-1})$. By (\ref{eq:link-r}), 
\[
|\pi(x_k)-\pi(x_{k-1})|\leq d_1<|\pi(x_k)-\pi(x_{k-2})|.
\] 
Therefore,  $\pi(x_k)>\pi(x_{k-1})$. This shows that $\{ \pi(x_k)\}_{k=0}^{k_\max}=\{\pi({r_1^*}^k(x))\}_{k=0}^{k_\max}$ is strictly increasing. 
Moreover 
\[
|\pi(1)-\pi(x_{k})|\leq d_1<|\pi(1)-\pi(x_{k-1})|.\]
Therefore $\pi(x_k)<\pi(1)$, and we conclude that the sequence $\{\pi(x_k)\}_{k=0}^{k_\max +1}$ is strictly increasing.  If $\pi(x_1)<\pi(x_0)$, then an analogous argument shows that the sequence is strictly decreasing. 

Now by Assumption \ref{ass:pi-1}, $\pi (r_1^*(0))>\pi(0)$, so if we take $x=0$, then the sequence is strictly increasing. This implies that $\pi(0) < \pi (1)$. Since for each choice of $x$, $x_0=0$ and $x_{k_{\max}+1}=1$, we have that, for each choice of $x$, the sequence $\{\pi(x_k)\}_{k=0}^{k_\max +1}$ is increasing. 

\end{proof}

{\noindent {\bf Remark.} The argument in the proof of the previous lemma still holds if  we remove the requirement that a uniform embedding $\pi$ should be injective. If $\pi$ is a uniform embedding which is not injective, then the argument shows that $\pi$ must be monotone, but not necessarily strictly monotone.}

\begin{lemma}\label{claim:r1}
Let $w$ and $\pi$ be as in Assumption \ref{ass:pi-1}, and $r_1^*(0)<\ell_1^*(0)$. Let $x\in[0,1]$ and $k_\max$ be the largest positive integer with ${r_1^*}^{k_\max}(x)<1$.
For every $0\leq k\leq k_\max-1$, 
\[
\pi([{r_1^*}^k(x), {r_1^*}^{k+1}(x)])\subseteq [\pi({r_1^*}^k(x)), \pi({r_1^*}^{k+1}(x))].\]
In addition, if $y\in[{r_1^*}^{k_\max}(x),1]$ then $\pi({r_1^*}^{k_\max}(x))\leq \pi(y)\leq\pi(1)$.
\end{lemma}
\begin{proof}

By our assumption we have $r_1^*(0)<\ell_1^*(1)$. Therefore

\[
\dom(r_1^*)\cap\dom(\ell_1^*)=[0,\ell_1^*(1)]\cap[r_1^*(0),1]=[r_1^*(0),\ell_1^*(1)].
\]

For $1\leq k\leq k_\max-2$, let $x_k={r_1^*}^{k}(x)$. Suppose $y\in(x_k,x_{k+1})$. Then $y\in\dom(r_1^*)\cap\dom(\ell_1^*)$ and we have $r_1^*(y)>x_{k+1}$ and $\ell_1^*(y)<x_k$.
From Lemma \ref{lemma:chain}, we know that $\pi$ is increasing on $\{ x_k\}_{k=1}^{k_\max}$, in particular 
$\pi(x_k)<\pi(x_{k+1})$. Moreover $\pi(\ell_1^*(y))<\pi(y)<\pi(r_1^*(y))$.
To satisfy the inequalities listed in (\ref{eq:link-r}) and (\ref{eq:link-l}), we restrict the location of $\pi(y)$. Namely $x_k<x_{k+1}<r_1^*(y)$, so 
\[
|\pi(x_{k+1})-\pi(r_1^*(y))|\leq d_1<|\pi(x_k)-\pi(r_1^*(y))|.
\]
Moreover, $|\pi(y)-\pi(r_1^*(y))|\leq d_1$, so $\pi(x_k)<\pi(y)$. 

Similarly, $\ell_1^*(y)<x_{k}<x_{k+1}$, and thus 
\[
|\pi(x_k)-\pi(\ell_1^*(y))| \leq d_1<|\pi(x_{k+1})-\pi(\ell_1^*(y))|.
\]
Also, $|\pi(y)-\pi(\ell_1^*(y))|\leq d_1$, and so $\pi(y)<\pi(x_{k+1})$. Thus $\pi(y)$ must belong to $(\pi(x_k), \pi(x_{k+1}))$. 

Now let $k=k_\max-1$ and $y\in(x_k,x_{k+1})$. Then either $y\leq\ell_1^*(1)$ or $y>\ell_1^*(1)$. For the case $y\leq\ell_1^*(1)$ by a similar discussion we obtain $\pi(y)\in (\pi(x_k),\pi(x_{k+1}))$. The case $y>\ell_1^*(1)$ can be dealt with in a similar fashion, but using 1 to take the place of $r_1^*(y)$. The remaining cases,  $y\in(0,x_1)$ and $y\in(x_{k_\max},1)$, can be done similarly.
\end{proof}

\begin{corollary}\label{cor:increasing}
If $w$ has a uniform linear embedding $\pi$ as in Assumption \ref{ass:pi-1} and $r_1^*(0)<\ell_1^*(1)$, then 
$\pi$ is strictly increasing. In particular, $\pi$ is continuous on all except countably many points in $[0,1]$.
\end{corollary}
\begin{proof}
Let $0\leq x<y\leq 1$. Let $k_{\max}$ be the largest non-negative integer with ${r_1^*}^{k_{\max}}(x)<1$. For $0\leq k\leq k_{\max}$, let $x_k={r_1^*}^k(x)$. By Lemma \ref{lemma:chain}, we know that $\pi$ is increasing on $\{x_k\}_{k=1}^{k_{\max}}$. Let $l$ be the largest integer (possibly zero) such that $x_l\leq y$.  If $l<k_{\max}$, then $y\in [x_l,x_{l+1}]$, and by Lemma \ref{claim:r1}, $\pi(y)\in[\pi(x_l),\pi(x_{l+1})]$. Therefore, $\pi(y)\geq\pi(x_l)\geq \pi(x)$, and since $\pi$ is an injection, $\pi(y)>\pi(x)$. If $l=k_{\max}$, then $y\in[x_l,1]$ and again by Lemma \ref{claim:r1},  it follows that $\pi(y)>\pi(x)$. So, $\pi$ is strictly  increasing. Consequently, any discontinuity in $\pi$ is of the form of a jump, and the sum of such jumps must be at most $\pi(1)$. Thus, $\pi$ is discontinuous in at most countably many points. 

\end{proof}
{ We also have monotonicity of $\pi$ when $\ell_1^*(1)<r_1^*(0)$. The proof of this fact can be found in \cite{hoda-thesis}, Chapter 4, Section 4.3, Lemmas 4.3.5--4.3.7. The proof is straightforward but rather technical, and therefore has been omitted here.}

\begin{definition}
Let $\pi$ be as in Assumption \ref{ass:pi-1}. Let 
\begin{eqnarray*}
\pi^+(x)&=\inf\{\pi(z):\, z>x\} &\text{for }x\in[0,1) \text{, and} \\
\pi^-(x)&=\sup\{\pi(z):\, z<x\} &\text{for }x\in (0,1].
\end{eqnarray*}
Note that these are limits of $\pi$ at $x$ from right and left respectively.
\end{definition}

\begin{proposition}\label{prop:pi-properties}
If $w$ has a uniform linear embedding $\pi$ as in Assumption \ref{ass:pi-1}, then 
 \begin{itemize}
 \item[(i)] If $\pi$ is continuous at  $x\in (0,\ell_i^*(1))$ then $\pi(r_i^*(x))=\pi(x)+d_i$. Likewise, if $\pi$ is continuous at  $x\in (r_i^*(0),1)$ then $\pi(\ell_i^*(x))=\pi(x)-d_i$.
 \item[(ii)] For all $x$ for which the limits are defined, we have that
 \begin{eqnarray*}
 \pi^{+}(\ell_i^*(x))&=\pi^{+}(x)-d_i\text{ and }\pi^{+}(r_i^*(x))&=\pi^{+}(x)+d_i, \text{ and}\\
 \pi^{-}(\ell_i^*(x))&=\pi^{-}(x)-d_i\text{ and }\pi^{-}(r_i^*(x))&=\pi^{-}(x)+d_i.
 \end{eqnarray*}
 
 \item[(iii)] If $\pi$ is continuous (from both sides) at $x\in (0,\ell^*(1))$ then  $\pi$ is also continuous  at $r_i^*(x)$. Likewise, if $\pi $ is continuous at $x\in (r^*_i(0),1)$, then $\pi$ is also continuous  at  $\ell_i^*(x)$.
 \end{itemize}
\end{proposition}
\begin{proof}
Part (i) follows from Proposition \ref{prop:cts-case} and the fact  that $\pi$ is increasing. 
We now prove $\pi^+(\ell_i^*(x))=\pi^+(x)-d_i$. The proof of the rest is similar. Let $z_n$ be a decreasing sequence converging to $x$ from the right. Then $\ell_i^*(z_n)$ is a decreasing sequence converging  to $\ell_i^*(x)$ from the right, which implies that $\pi(z_{n-1})-\pi(\ell_i^*(z_n))>d_i$. On the other hand, $\pi(z_n)-\pi(\ell_i^*(z_n))\leq  d_i$. Taking limits from both sides of these inequalities, we get
$\lim_{n\rightarrow \infty}\pi(\ell_i^*(z_n))=\lim_{n\rightarrow \infty }\pi(z_n)-d_i$.
This gives us that $\pi^+(\ell_i^*(x))=\pi^+(x)-d_i$, since $\pi$ is increasing. 
Part (iii) follows from (ii) immediately. 
\end{proof}

\begin{remark}\label{remark:pi-cts-0}
\begin{enumerate}
\item In part (iii)  of Proposition \ref{prop:pi-properties}, neither $x$ nor $r_i^*(x)$ (respectively $\ell_i^*(x)$) can be 0 or 1, since we need to compare continuity at these points from both sides. It is easy to construct examples where $\pi$ is continuous (from the right) at 0 but not continuous at  $r_1^*(0)$. 

\item Let $\pi$ be a uniform linear embedding for  $w$ as in Assumptions \ref{ass:pi-1}. By Corollary \ref{cor:increasing}, $\pi$ is increasing. 
We can assume without loss of generality that $\pi$ is continuous at 0. Indeed, if $0=\pi(0)<\pi^+(0)$, then define the new function $\pi'$ to be
$\pi'(x)=\pi(x)-\pi^+(0)$ when $x\neq 0$ and $\pi'(0)=0$. It is clear that $\pi'$ forms a uniform linear embedding for $w|_{(0,1]}$. To show that $\pi'$ is a uniform linear embedding for $w$, let $x>0$, and note that there exists a decreasing sequence $\{z_n\}$ converging to 0 such that $w(x,0)=w(x,z_n)=\alpha_i$. (This is true because $w$ is well-separated). So for every $n\in {\mathbb N}$, 
we have $d_{i-1}<\pi(x)-\pi(z_n)\leq d_i$. Therefore, $d_{i-1}<\pi(x)-\pi^+(0)\leq d_i$, {\it i.e.} $d_{i-1}<\pi'(x)-\pi'(0)\leq d_i$. Therefore, we can assume that $\pi$ is continuous at 0.
\end{enumerate}
\end{remark}


\subsection{Properties of the displacement function}\label{subset:displacement-necessary}

We now collect some facts about the displacement function  as given in Definition \ref{def:displacement}, given that a uniform linear embeddding exists. In particular, we give the relation between $\delta$ and the limit behaviour of $\pi$. Recall that for $x\in(0,1)$, $\pi^+(x)$ and $\pi^-(x)$ are the right and left limits of $\pi$ at the point $x$.
In the proof of Proposition \ref{prop:pi-delta}, we use this simple observation: a point $x\in (0,1)$ belongs to the interior of the domain of a legal function $\phi$ if and only if $\phi(x)$ never touches either of 0 or 1. More precisely, if $\dom(\phi)=[p,q]$, then $x\in(p,q]$ (respectively $x\in[p,q)$) precisely when $\phi(x)$ does not touch 0 (respectively 1).
\begin{proposition}\label{prop:pi-delta}
Let $w$ and  $\pi$ satisfy Assumption \ref{ass:pi-1}. Let  $\phi=f_1\circ\ldots\circ f_k$ be a legal function with domain $[p,q]$, where $p\in \P$ and $q\in \Q$.  Then
\begin{itemize}
\item[(i)] $\pi^+(\phi(x))-\pi^+(x)=\delta(\phi)$ for every $x\in[p,q)$. 
\item[(ii)] $\pi^-(\phi(x))-\pi^-(x)=\delta(\phi)$ for every $x\in(p,q]$.  
\item [(iii)] If $\psi_1$ and $\psi_2$ are two legal functions with  $\{x\}\subsetneq\dom(\psi_1)\cap\dom(\psi_2)$, and   $\psi_1(x)=\psi_2(x)$ then $\delta(\psi_1)=\delta(\psi_2)$.
\item [(iv)] If $\delta(\phi)=0$ and domain of $\phi$ is not a singleton, then  $\phi$ is the identity function on its domain. 
\end{itemize}
\end{proposition}
\begin{proof}
First note that $\phi$ is strictly increasing, therefore $\phi$ applied to $x\in[p,q)$ never touches 1, {\it i.e.} $f_l\circ\ldots\circ f_k(x)\neq 1$ for every $1\leq l\leq k$. 
So, we can apply Proposition \ref{prop:pi-properties} (ii) in stages to obtain $\pi^+(\phi(x))=\delta(\phi)+\pi^+(x)$. Note that the condition that ``$\phi$ never touches 1'' guarantees that $\pi^+$ is defined in every step.
This finishes the proof of (i). We skip the proof of (ii) as it is similar to (i). 

The third statement is an easy corollary of (i), as $\psi_1^{-1}\circ\psi_2$ applied to $x$ either never touches 0 or never touches 1. To see this fact, it is enough to observe that  if a legal function applied to $x$ touches both 0 and 1 then its domain is the singleton $\{x\}$. Assume, without loss of generality that $\psi_1^{-1}\circ\psi_2(x)$ does not touch 1. Therefore, $\pi^+(\psi_1^{-1}\circ\psi_2(x))-\pi^+(x)=\delta(\psi_1^{-1}\circ\psi_2)$. Thus, $0=\delta(\psi_1^{-1}\circ\psi_2)=-\delta(\psi_1)+\delta(\psi_2)$, and we are done.

To prove (iv), assume that $\dom(\phi)$ is not a singleton. 
Let $z$ be a point in the interior of $\dom(\phi)$. By (i), we have $\pi^+(\phi(z))-\pi^+(z)=\delta(\phi)=0$. Also note that $\pi^+$ is strictly increasing, as $\pi$ is. Therefore, we conclude that $\phi(z)=z$ on every point in the interior of $\dom(\phi)$. This finishes the proof of (iv), since $\phi$ is continuous and its domain is just a closed interval. 
\end{proof}
The following corollary is merely a restatement of Proposition \ref{prop:pi-delta} part (iv).
\begin{corollary}\label{cor:equal-delta}
Let $\phi$ and $\psi$ be legal functions such that $\dom(\phi)\cap\dom(\psi)$ is nonempty and non-singleton. If $\delta(\phi)=\delta(\psi)$ then $\phi=\psi$ on the intersection of their domains. 
\end{corollary}
In what follows we extend Proposition \ref{prop:pi-delta} to the cases where a legal function has a singleton as its domain. Let us start with an auxiliary lemma.
\begin{lemma}\label{lem:equal-jumps}
Let $\pi$ and $w$ be as in Assumption \ref{ass:pi-1}. Let $1\leq i<j\leq N-1$ be fixed.  Then $\pi^+(r_i^*(0))-\pi^-(r_i^*(0))=\pi^+(r_j^*(0))-\pi^-(r_j^*(0))$.
\end{lemma}
\begin{proof}
Let $\phi$ be a legal function such that $\phi(r_i^*(0))$ does not touch 0 or 1. Thus, by Proposition \ref{prop:pi-delta}, we have
$
\pi^+(\phi(r_i^*(0)))-\pi^-(\phi(r_i^*(0)))=\pi^+(r_i^*(0))-\pi^-(r_i^*(0)).
$
Thus,  
\begin{eqnarray*}
\pi^+(r^*_j\circ r_i^*(0))-\pi^-(r_j^*\circ r_i^*(0))&=&\pi^+(r_i^*(0))-\pi^-(r_i^*(0)),\\
\ \pi^+(r^*_i\circ r_j^*(0))-\pi^-(r_i^*\circ r_j^*(0))&=&\pi^+(r_j^*(0))-\pi^-(r_j^*(0)).
\end{eqnarray*}
Moreover by Proposition \ref{prop:pi-delta} (i), we have
$$\pi^+(r_i^*\circ r_j^*(x))=\pi^+(r_j^*\circ r_i^*(x))=d_i+d_j+\pi(x),$$
if $r_i^*\circ r_j^*(x)<1$ and $r_j^*\circ r_i^*(x)<1$.
This implies that $r_i^*\circ r_j^*=r_j^*\circ r_i^*$ on the intersection of their domains, as $\pi^+$ is strictly increasing. By assumption, we are not in the case of Exception \ref{SpecialCase}, and thus we have that $r_1^*\circ r_j^*(0)< 1$ and $r_j^*\circ r_1^*(0)<1$, for all $1< j\leq N-1$. Therefore, $r_1^*\circ r_j^*(0)=r_j^*\circ r_1^*(0)$, and so
$\pi^+(r_1^*(0))-\pi^-(r_1^*(0))=\pi^+(r_j^*(0))-\pi^-(r_j^*(0)).$ This implies that $\pi^+(r_i^*(0))-\pi^-(r_i^*(0))=\pi^+(r_j^*(0))-\pi^-(r_j^*(0)),$ where $1\leq i<j\leq N-1$. 
\end{proof}

\begin{proposition}\label{singleton-domain-delta}
Under Assumption \ref{ass:pi-1}, for legal functions $\phi$ and $\psi$ with $\dom(\phi)=\dom(\psi)=\{x_0\}$ the following holds:
$\phi(x_0)=\psi(x_0)$ if and only if $\delta(\phi)=\delta(\psi)$.
\end{proposition}
\begin{proof}
We begin the proof by considering some special cases.
First, suppose $\phi$ and $\psi$ are legal functions with $\phi(0)=\psi(0)=1$, such that $\phi(0)$ and $\psi(0)$  do not touch 0 or 1 at any intermediate step. In that case, for some $i,j, i',j'$ we have
$$\phi=r_i^*\circ \phi_1\circ r_{i'}^* \ \mbox{ and }\  \psi=r_j^*\circ \psi_1\circ r_{j'}^*,$$
such that $\phi_1(z_1)$ and $\psi_1(z_2)$ do not touch 0 or 1, where $z_1=r_{i'}^*(0)$ and $z_2=r_{j'}^*(0)$. 
By Proposition \ref{prop:pi-delta}, 
$\pi^+(\phi_1(r^*_{i'}(0)))-\pi^+(r^*_{i'}(0))=\delta(\phi_1)\ \mbox{and }\ \pi^-(1)-\pi^-(\phi_1(r^*_{i'}(0)))=d_i.$
Recall that $\pi$ is assumed to be continuous at 0, {\it i.e.} $\pi^+(0)=0$. So, $\pi^+(r^*_{i'}(0))=d_{i'}$.
Thus, 
$$\delta(\phi)=d_i+\delta(\phi_1)+d_{i'}=\pi^+(\phi_1\circ r_{i'}^*(0))-\pi^-(\phi_1 \circ r_{i'}^*(0))+\pi^-(1).$$ 
Similarly, $\delta(\psi)=\pi^+(\psi_1\circ r^*_{j'}(0))-\pi^-(\psi_1\circ r^*_{j'}(0))+\pi^-(1)$. 
Combining (i) and (ii) of Proposition \ref{prop:pi-delta}, we observe that $\pi^+(\psi_1\circ r^*_{j'}(0))-\pi^-(\psi_1\circ r^*_{j'}(0))=\pi^+(r^*_{j'}(0))-\pi^-(r^*_{j'}(0))$, as $\psi_1$ applied to $r_{j'}^*(0)$ does not hit 0 or 1. Similarly, $\pi^+(\phi_1\circ r_{i'}^*(0))-\pi^-(\phi_1 \circ r_{i'}^*(0))=\pi^+(r_{i'}^*(0))-\pi^-(r_{i'}^*(0))$.
Applying Lemma \ref{lem:equal-jumps}, we conclude that 
\begin{equation}\label{eq:0-1}
\delta(\phi)=\delta(\psi)=\pi^-(1)+\pi^+(r_k^*(0))-\pi^-(r_k^*(0))\ \mbox{ for any } 1\leq k\leq N-1.
\end{equation}
Similarly, suppose $\phi$ and $\psi$ are legal functions with $\phi(1)=\psi(1)=0$, such that $\phi(1)$ and $\psi(1)$  do not touch 0 or 1 at any intermediate step. Then $\phi^{-1}$ and $\psi^{-1}$
satisfy the conditions of the previous case, and we have $\delta(\phi^{-1})=\delta(\psi^{-1})$. So, 
\begin{equation}\label{eq:1-0}
\delta(\phi)=\delta(\psi)=-\pi^-(1)+\pi^-(r_k^*(0))-\pi^+(r_k^*(0))\ \mbox{ for any } 1\leq k\leq N-1.
\end{equation}
Next, consider the case where $\phi(0)=\psi(0)=0$, such that $\phi(0)$ and $\psi(0)$  do not touch 0 or 1 at any intermediate step. Thus, by Proposition \ref{prop:pi-delta}, we have $\delta(\phi)=\delta(\psi)=0$. A similar argument works when $\phi(1)=\psi(1)=1$, and  $\phi(1)$ and $\psi(1)$  do not touch 0 or 1 at any intermediate step.

We can now prove the general case. Suppose $\dom(\phi)=\dom(\psi)=\{x_0\}$ and $\phi(x_0)=\psi(x_0)$. Recall that $x_0\in\P$, and let $\eta$ be a legal function such that $\eta(0)=x_0$. 
Define
$$\xi=\eta^{-1}\circ \psi^{-1}\circ \phi\circ \eta.$$
Clearly, $\xi(0)=0$. Let $\xi=\xi_1\circ\ldots\circ \xi_n$ be a decomposition of $\xi$ into legal functions, where each $\xi_i$ satisfies the conditions of one of the cases studied above. That is, 
the domain and the range of each $\xi_i$ is either $\{0\}$ or $\{1\}$, and none of the $\xi_i$'s touch 0 or 1 in any intermediate step. 
It is easy to observe that the number of terms $\xi_i$ which map 0 to 1 must be the same as the number of terms which map $1$ to 0. Therefore, by what we observed above, 
$\delta(\xi)=0$. 
This proves the ``only if'' direction. 

To prove the ``if'' direction, it is enough to prove that if $\delta(\phi)=0$ then $\phi(x_0)=x_0$.   We consider two possibilities: Firstly, suppose that there exists a legal function $\eta$ with $\eta(0)=x_0$ such that $\eta(0)$ does not touch 1. Let $\xi=\phi\circ\eta$, and note that $\dom(\xi)=\{0\}$. Moreover, $\delta(\xi)=\delta(\eta)$, as $\delta(\phi)=0$. By Proposition \ref{prop:orbits-domain} (ii), $\xi(0)$ must touch 1. Suppose $\xi$ decomposes into $\xi=\xi_1\circ\xi_2\circ \xi_3$, where
$\xi_1(0)=\phi(x_0)$, $\xi_2(1)=0$, $\xi_3(0)=1$, and $\xi_1(0)$ does not touch 1. Applying part (i) to $\xi_2$ and $\xi_3^{-1}$, we conclude that $\delta(\xi_2)=-\delta(\xi_3)$.
Therefore $\delta(\xi_1)=\delta(\eta)$. Moreover, $\{0\}\subsetneq \dom(\xi_1)\cap \dom(\eta)$. Thus, by Corollary \ref{cor:equal-delta} we have $\xi_1(0)=\eta(0)$, which shows that $\phi(x_0)=x_0$ in this case. 

Next, suppose that $\xi$ decomposes to $\xi=\xi_1\circ\xi_2$ where $\xi_2(0)=1$, $\xi_1(1)=\phi(x_0)$, and $\xi_1(1)$ does not touch 0.
As argued in the proof of Lemma \ref{lem:equal-jumps}, we have $\delta(\xi_2)=\pi^-(1)+\pi^+(r_1^*(0))-\pi^-(r_1^*(0))$. Thus,
$$\pi^+(x_0)=\delta(\eta)=\delta(\xi)=\pi^-(1)+\pi^+(r_1^*(0))-\pi^-(r_1^*(0))+\pi^-(\phi(x_0))-\pi^-(1).$$
This implies that $\pi^+(x_0)\geq \pi^-(\phi(x_0))$.
Observe that  $x\leq y$ if and only if $\pi^-(x)\leq \pi^+(x)\leq \pi^-(y)\leq \pi^+(y)$,  since $\pi$ is increasing. Hence, we must have $\pi^+(x_0)\geq \pi^+(\phi(x_0))$. Since $\pi^+$ is increasing, 
we conclude that $x_0\geq \phi(x_0)$. 
Replacing $\phi$ by $\phi^{-1}$, we obtain $x_0\geq \phi^{-1}(x_0)$ in a similar way. Hence, $x_0=\phi(x_0)$ in this case as well.

Secondly, suppose that there exists a legal function $\eta$ with $\eta(1)=x_0$ such that $\eta(1)$ does not touch 0. A similar argument, where the roles of 0 and 1 are switched, finishes the proof.
\end{proof}

\begin{remark}
The sets $\P$ and $\Q$ are  either disjoint or identical. Namely, $\P\cap \Q$ is nonempty only when there exists a legal composition $\phi$ with $\phi(0)=1$. Indeed, assume that there exists an element $x\in\P\cap \Q$. 
Let $\phi$ and $\psi$ be legal compositions such that
\begin{eqnarray*}
x=\phi(0)=\psi(1).
\end{eqnarray*}
By Observation \ref{obs1}, $\psi$ is invertible, and 
$$1=\psi^{-1}\circ\phi(0).$$
Therefore, $1\in\P$, which in turn implies that $\Q\subseteq \P$. Similarly, we observe that $0\in \Q$ as well, and $\P\subseteq \Q$. 

Clearly if $\P=\Q$ and $\phi(0)=1$, where $\phi$ is a legal function, then $\dom(\phi)=\{0\}$. On the other hand, when $\P\cap\Q=\emptyset$, no legal function $\psi$ has a singleton as its domain, because $\dom(\psi)=[p,q]$ with $p\in\P$ and $q\in\Q$.
\end{remark}


\subsection{Conditions of Theorem \ref{thm:nec-suff-cond-N} are necessary.}
We are now ready to prove the necessity of the conditions of Theorem \ref{thm:nec-suff-cond-N}. 
\begin{proof}
Suppose $w$ and $\pi$ satisfy Assumption \ref{ass:pi-1}. In particular, recall that $\pi$ is strictly increasing. 
If $\P\cap\Q=\emptyset$ then an easy application of Proposition \ref{prop:pi-delta} shows that Conditions (1), (2a), and (2b) must hold.
Indeed, let  $\phi=f_1\circ\ldots\circ f_k$ be a legal function so that $\phi(x)=x$ for some $x$. Clearly, $x$ belongs to the domain of either $\ell^*_1\circ r_1^*$ or $r_1^*\circ \ell_1^*$, which are both identity functions on their domains with zero displacement. So by Proposition \ref{prop:pi-delta}, we have $\delta(\phi)=0$, and therefore $\phi$ must be the identity function on its domain. 
To prove (2a), suppose that $x=\psi(0)$ and $y=\phi(0)$ where $x<y$, and $\psi$ and $\phi$ are legal functions. The functions $\psi$ and $\phi$ applied to 0 never touch 1, since $\P\cap\Q=\emptyset$. So $\pi^+(\phi(0))-\pi^+(0)=\delta(\phi)$ and $\pi^+(\psi(0))-\pi^+(0)=\delta(\psi)$. This finishes the proof of part (2a), since $\pi^+$ is a strictly increasing function, and $\pi^+(0)=0$.

Suppose now that $y=\phi(1)<r_i^*(0)$ for some $i$ and $x\in\P$ with $x=\psi(0)$. By Proposition \ref{prop:pi-delta} we have 
\[
\pi^-(y)-\pi^-(1)=\delta (\phi)\quad \mbox{and} \quad \pi^+(x)-\pi^+(0)=\delta (\psi).
\]
On the other hand by Corollary \ref{cor:increasing}, $\pi^-(y)\leq\pi(y)<\pi(r_i^*(0))\leq d_i$. Hence $\pi^-(y)-\pi^-(1)<d_i-\pi^-(1)$, which gives $\delta(\phi)<d_i-\pi^-(1)$. Moreover by definition of $\pi^+$ and $\pi^-$ we have $\pi^+(x)<\pi^-(1)$. Note that $\pi^+(0)=0$.
This implies that $\delta(\psi)<\pi^-(1)$ and therefore, $\delta(\psi)<\pi^-(1)< d_i-\delta(\phi)$. Taking $a=\pi^-(1)$  finishes the proof of (2b).

Next, consider the case where $\P=\Q$. Condition (1) holds by Proposition \ref{prop:pi-delta}. Now suppose that $x=\psi(0)$ and $y=\phi(0)$ where $x<y$, and $\psi$ and $\phi$ are legal functions. If domains of $\phi$ and $\psi$ are not singletons, then Condition (2a)  follows from  Proposition \ref{prop:pi-delta}.
Now assume that $\dom(\psi)=\{0\}$, which means that $\psi(0)$ must touch 1. Using Proposition \ref{singleton-domain-delta} and Equations (\ref{eq:0-1}) and (\ref{eq:1-0}), note that one of the following two cases happen: Either $\psi(0)$ does not touch 0 after the last time it touches 1, in which case we have 
$\delta(\psi)=\pi^+(r_1^*(0))-\pi^-(r_1^*(0))+\pi^-(x)$. Or, $\psi(0)$ touches 0 after the last time it touches 1, in which case we have $\delta(\psi)=\pi^+(x)$. (See the proof of Proposition \ref{singleton-domain-delta} part (ii) for a similar argument). In the latter case, $\psi$ can be decomposed as $\psi_1\circ\psi_2$, where $\psi_2(0)=0$ and $\psi_1(0)$ does not touch 1. Thus,  thanks to Lemma \ref{lem:equal-jumps}, we have $\pi^+(\psi_1(0))-\pi^-(\psi_1(0))=\pi^+(r_1^*(0))-\pi^-(r_1^*(0))$. Clearly, $x=\psi(0)=\psi_1(0)$, and so
$$\pi^+(r_1^*(0))-\pi^-(r_1^*(0))+\pi^-(x)=\pi^+(x)$$
equals the value of $\delta(\psi)$ in either case. Thus $\delta(x)<\delta(y)$ as $\pi^+$ is strictly increasing. 
\end{proof}

\section{Sufficiency; construction of a uniform linear embedding}\label{sec:sufficient}

In this last section we will show that if Conditions (1), (2a) and (2b) of Theorem \ref{thm:nec-suff-cond-N} hold, then a uniform linear embedding exists. Moreover, we will give an explicit construction for the uniform linear embedding. Throughout this section we make the following assumption. 

\begin{assumption}\label{assump:suff}
Let $w$ be as in Assumption \ref{ass:w}.  Assume that the conditions of Theorem \ref{thm:nec-suff-cond-N} hold. In particular, 
\begin{enumerate}
\item If $\phi$ is a legal function with $\phi(x)=x$ for some $x\in \dom(\phi)$, then $\phi$ is the identity function on its domain. 
\item  There exist real numbers $0<d_1<\ldots<d_{N-1}$ such that 
\begin{itemize}
\item[(2a)] For all $x,y\in\P$, and legal compositions $\phi,\psi$ so that $x=\phi(0)$ and $y=\psi(0)$, we have that, if  $x<y$ then $\delta(\phi)<\delta(\psi)$.
\item[(2b)] If $\P\cap\Q=\emptyset$ then there exists $a\in \mathbb{R}^{\geq 0}$ which satisfies the following condition:
If $\phi$ and $\psi$ are legal compositions with $1\in\dom(\phi)$ and $0\in \dom(\psi)$, and  if $\phi(1)<r_i^*(0)$  then 
$\delta(\psi)<a<d_i-\delta(\phi)$. 
\end{itemize}
\end{enumerate}
\end{assumption}

\subsection{The displacement function in the case where $\P\cap\Q=\emptyset$.}
This subsection contains results regarding the displacement function $\delta$ under Assumption \ref{assump:suff}, which we will use to prove sufficiency of the conditions of Theorem 
\ref{thm:nec-suff-cond-N} in the case where $\P\cap\Q=\emptyset$. Our objective is to show that, when $\P\cap\Q=\emptyset$, there is a nice one-to-one correspondence between the sets $\P$ and $\Q$.

\begin{lemma}[$\delta$ is ``increasing'']\label{lem:delta-increasing}
Under Assumption \ref{assump:suff} and  $\P\cap\Q=\emptyset$,  let $\phi_1$ and $\phi_2$ be legal functions  and $x\in \dom(\phi_1)\cap \dom(\phi_2)$.  Then
\begin{itemize}
\item[(i)] If $\phi_1(x)<\phi_2(x)$ then $\delta(\phi_1)<\delta(\phi_2).$
\item[(ii)] If $\phi_1(x)=\phi_2(x)$ then $\delta(\phi_1)=\delta(\phi_2)$.
\end{itemize}
\end{lemma}
\begin{proof}
Let $\phi_1$ and $\phi_2$ be as above. By Lemma \ref{lemma:domains-of-legal},  $\dom(\phi_1)=[p_1,q_1]$ and $\dom(\phi_2)=[p_2,q_2]$ for  $p_1,p_2\in\P$ and $q_1,q_2\in\Q$. Since $\dom(\phi_1)\cap \dom(\phi_2)\neq\emptyset$ we have $\dom(\phi_1)\cap \dom(\phi_2)=[p,q]$ with $p=\max\{p_1,p_2\}$ and $q=\min\{q_1,q_2\}$. If $x=p$ then Condition (2a) of Theorem \ref{thm:nec-suff-cond-N} implies that $\delta(\phi_1)<\delta(\phi_2)$ and we are done. Now suppose $x\in(p,q]$. We will show that $\phi_1(p)<\phi_2(p)$. Suppose to the contrary that $\phi_1(p)\geq \phi_2(p)$. Since $\phi_1$ and $\phi_2$ are continuous functions on $[p,q]$, and $\phi_1(x)<\phi_2(x)$,  there exists a point $y\in [p,x)$ with $\phi_1(y)=\phi_2(y)$. Hence $\range(\phi_1)\cap\range(\phi_2)\neq \emptyset$. So by Lemma \ref{lemma:domains-of-legal}, $\range(\phi_1)\cap\range(\phi_2)=[p',q']$ where $p'=\max\{\phi_1(p_1),\phi_2(p_2)\}\in\P$ and $q'=\min\{\phi_1(q_1),\phi_2(q_2)\}\in\Q$. In particular, since $\P\cap \Q=\emptyset$, we have
\begin{equation*}
 \max\{p_1,p_2\}<\min\{q_1,q_2\}  \ \ \mbox{ and }\ \ \max\{\phi_1(p_1),\phi_2(p_2)\}<\min\{\phi_1(q_1),\phi_2(q_2)\}.
\end{equation*}
Clearly $\phi_2^{-1}\circ\phi_1(y)=y$. Thus by Condition (1) of Theorem \ref{thm:nec-suff-cond-N}, we have $\phi_2^{-1}\circ \phi_1$ is the identity function on its domain.
Next we observe that $x\in  \dom(\phi_2^{-1}\circ\phi_1)$. This follows from the fact that $\phi_1(x)\in \range(\phi_2)=\dom(\phi_2^{-1})$, as 
$$\phi_2(p_2)\leq \phi_2(y)=\phi_1(y)<\phi_1(x)<\phi_2(x)\leq \phi_2(q_2).$$ 
Therefore, we must have $\phi_2^{-1}\circ\phi_1(x)=x$, {\it i.e.} $\phi_1(x)=\phi_2(x)$, which is a contradiction. This finishes the proof of (i).

To prove (ii), let $x\in\dom(\phi_1)\cap\dom(\phi_2)$ be such that $\phi_1(x)=\phi_2(x)$. Assume without loss of generality that $\delta(\phi_2)\geq \delta(\phi_1)$, thus $\delta(\phi_1^{-1}\circ\phi_2)\geq 0$.  Clearly $\phi_1^{-1}\circ\phi_2(x)=x$, and by Condition (1) of Theorem 
\ref{thm:nec-suff-cond-N}, we have that the legal function $\phi_1^{-1}\circ\phi_2$ is the identity function on its domain, say $[p,q]$. In particular, $\phi_1^{-1}\circ\phi_2(p)=p$.
Let $\eta$ be a legal function with $\eta(0)=p$, and consider $\psi=\eta^{-1}\circ\phi_1^{-1}\circ\phi_2\circ\eta$. Clearly $\psi^n(0)=0<r_1^*(0)$, where $\psi^n$ denotes the $n$-fold composition of $\psi$ with itself. Now Condition (2a) of Theorem  \ref{thm:nec-suff-cond-N} gives that 
$$0\leq n\delta(\phi_1^{-1}\circ\phi_2)=\delta(\psi^n)<d_1,$$
for every positive integer $n$. So $\delta(\phi_1^{-1}\circ\phi_2)$ must be zero, and 
 $\delta(\phi_1)=\delta(\phi_2)$. 
\end{proof}

\begin{corollary}\label{cor:h-delta-increasing}
Under Assumption \ref{assump:suff} and  $\P\cap\Q=\emptyset$, let $\phi_1$ and $\phi_2$  be legal functions which have non-disjoint domains. Then
\begin{itemize}
\item[(i)]  $\phi_1(x)<\phi_2(x)$ for some $x\in \dom(\phi_1)\cap \dom(\phi_2)$ if and only if $\delta(\phi_1)< \delta(\phi_2)$ if and only if $\phi_1<\phi_2$ everywhere on $\dom(\phi_1)\cap \dom(\phi_2)$.
\item[(ii)] $\phi_1(x)=\phi_2(x)$ for some $x\in \dom(\phi_1)\cap \dom(\phi_2)$ if and only if $\delta(\phi_1)= \delta(\phi_2)$ if and only if $\phi_1=\phi_2$ everywhere on $\dom(\phi_1)\cap \dom(\phi_2)$.
\end{itemize}
\end{corollary}

The following lemma is a direct consequence of Lemma \ref{lem:delta-increasing}.

\begin{lemma}\label{lem:use-for-1-1-correspondence}
Suppose that we are in the settings of Lemma \ref{lem:delta-increasing}. Let $\phi=f_1\circ \dots\circ f_k$ be a legal composition and $x\in\dom(\phi)$. 
For $1\leq t\leq k$, let $x_t=f_{k-t+1}\circ \dots \circ f_k(x)$. Then
\begin{itemize}
\item[(i)] If $x\leq\min\{x_t:\, 1\leq t\leq k\}$ then $[0,x]\subseteq \dom(\phi)$.
\item[(ii)] If $x\geq \max\{x_t:\, 1\leq t\leq k\}$ then $[x,1]\subseteq \dom(\phi)$.
\end{itemize}
\end{lemma}
\begin{proof}
We only prove (i). A similar argument proves part (ii) of the lemma.
Let $\phi$ and $x$ be as stated above. To prove part (i), for $1\leq t\leq k$, define $\phi_t=f_{k-t+1}\circ \dots \circ f_k$, and let $\phi_0$ be the identity function. Then for every $t$, $x_t=\phi_t(x)$, and specifically $x_0=x$ and $x_k=\phi(x)$. 

Assume $x\leq x_t$ for all $1\leq t\leq k$.
We show by induction on $t$ that $[0,x]\subseteq \dom(\phi_t)$. Since the domain of a legal function is an interval, and $x\in \dom(\phi_t)$, it is enough to show that $0\in \dom(\phi_t)$. For $t=0$, the statement is obvious. 
For the induction step, fix $t$, $0<t\leq k$, and consider $\phi_t=f_{k-t+1}\circ \phi_{t-1}$. By the induction hypothesis, $0\in \dom(\phi_{t-1})$. Let $y=\phi_{t-1}(0)$. We must show that $y\in \dom(f_{k-t+1})$. Consider two cases:
First assume that $f_{k-t+1}=\ell^*_i$ for some $1\leq i\leq N-1$. Since $x\leq x_t=\ell^*_i\circ\phi_{t-1}(x)$, we have $\phi_{t-1}(x)\geq r_i^*(x)$. 
So by Lemma \ref{lem:delta-increasing} we have $\delta(\phi_{t-1})\geq \delta(r_i^*)=d_i$. 
This, together with another application of Lemma \ref{lem:delta-increasing},  implies that $\phi_{t-1}(0)\geq r_i^*(0)$.  Hence $0\in \dom(\phi_t)$.

Next assume that $f_{k-t+1}=r_i^*$,  where $1\leq i\leq N-1$. Since $0\leq x$, we must have that $y\leq x_{t-1}$. 
Observe that  $\dom(f_{k-t+1})=[0,\ell_i^*(1)]$ and $x_{t-1}\in \dom(f_{k-t+1})$. 
Therefore $y\in \dom (f_{k-t+1})$, and hence $0\in \dom(\phi_t)$. This completes the proof of part (i). 
\end{proof}
\begin{corollary}\label{cor:dom-range-inclusion}
Under Assumption \ref{assump:suff} and $\P\cap\Q=\emptyset$, let  $\phi_1=f_1\circ \dots\circ f_k$ and $\phi_2=g_1\circ \dots\circ g_l$ be legal functions with $x\in \dom(\phi_1)\cap\dom(\phi_2)$. 
For $1\leq t\leq k$, let $x_t=f_{k-t+1}\circ \dots \circ f_k(x)$. Similarly, let  $y_t=g_{l-t+1}\circ \dots \circ g_l(x)$ for $1\leq t\leq l$.

\begin{itemize}
\item[(i)] Suppose $x\leq x_t\leq \phi_1(x)$ for every $1\leq t\leq k$, and $x\leq y_t\leq \phi_2(x)$ for every $1\leq t\leq l$. Then $\dom(\phi_i)=[0,\phi_i^{-1}(1)]$ and $\range(\phi_i)=[\phi_i(0),1]$. Moreover if $\phi_2(x)<\phi_1(x)$, then $\dom(\phi_1)\subset\dom(\phi_2)$ and $\range(\phi_1)\subset\range(\phi_2).$
\item[(ii)] Suppose $x\geq x_t\geq \phi_1(x)$ for every $1\leq t\leq k$, and $x\geq y_t\geq \phi_2(x)$ for every $1\leq t\leq l$. Then $\dom(\phi_i)=[\phi_i^{-1}(0),1]$ and $\range(\phi_i)=[0,\phi_i(1)]$. Moreover if $\phi_1(x)<\phi_2(x)$, then $\dom(\phi_1)\subset\dom(\phi_2)$ and $\range(\phi_1)\subset\range(\phi_2).$
\end{itemize}

\end{corollary}
\begin{proof}
We only prove (i). A similar argument proves part (ii). 
Let $\dom(\phi_1)=[p,q]$. By Corollary \ref{cor:h-delta-increasing}, for every $1\leq t\leq k$ we have $\delta(f_{k-t+1}\circ \dots\circ f_k)\geq 0$ and $\delta(f_{k-t+1}\circ \dots\circ f_k)\leq \delta(\phi_1)$, since $x\leq f_{k-t+1}\circ \dots\circ f_k(x)\leq \phi_1(x)$. Thus, $p\leq f_{k-t+1}\circ \dots\circ f_k(p)\leq \phi_1(p)$ and $q\leq f_{k-t+1}\circ \dots\circ f_k(q)\leq \phi_1(q)$ for every $1\leq t\leq k$. 
Now Proposition \ref{prop:orbits-domain} immediately implies that $p=0$ and $\phi_1(q)=1$. This proves the first part of (i).
Next, observe that since $\phi_2(x)<\phi_1(x)$, by Corollary \ref{cor:h-delta-increasing} we have $\delta(\phi_2)<\delta(\phi_1)$. This implies that $\delta(\phi_1^{-1})<\delta(\phi_2^{-1})$. Then $\phi_2(0)<\phi_1(0)$ and $\phi_1^{-1}(1)<\phi_2^{-1}(1)$. This completes the proof of the second part of (i).
\end{proof}

The next lemma states that each element of the set $\P$ is paired with exactly one element of the set $\Q$. 
\begin{lemma}\label{lem:new-1-1-correspondence}(One-to-one correspondence between $\P$ and $\Q$)
Let $w$ be as in Assumption \ref{assump:suff}, where $\P\cap\Q=\emptyset$. Then for any legal function $\phi$ with $0\in\dom(\phi)$ and signature $(m_1,\ldots,m_{N-1})$ there is a legal function $\psi$ with $1\in\dom(\psi)$ and signature $(-m_1,\ldots,-m_{N-1})$ and vice versa.
\end{lemma}
\begin{proof}
Suppose that $\phi$ is a legal function with signature $(m_1,\ldots,m_{N-1})$ and $0\in\dom(\phi)$. Let $\phi=\psi_k\circ\ldots\circ \psi_1$ be a decomposition of $\phi$ into legal functions $\psi_1,\ldots, \psi_k$ with the following property. Let $x_0=0, x_i=\psi_i\circ\ldots\circ \psi_1(0)$, $1\leq i\leq k$. The legal functions $\psi_1,\ldots, \psi_k$ are chosen such that $x_i=\psi_i(x_{i-1})$ is the maximum point of orbit of 
$x_{i-1}$ under $\psi_k\circ\ldots\circ \psi_i$ when $i$ is odd, and is the minimum point of orbit of $x_{i-1}$ under $\psi_k\circ\ldots\circ \psi_i$ when $i$ is even. 

We now prove that $\psi=\psi_k^{-1}\circ\ldots\circ \psi_1^{-1}$ is a legal function with $x_1\in \dom(\psi)$. From the definition of the $x_1$, we know that $x_1$ is the maximum point of the orbit of 0 under $\psi_k\circ\ldots\circ \psi_1$ Therefore, the maximum point of the orbit of $x_1$ under $\psi=\psi_k^{-1}\circ\ldots\circ \psi_1^{-1}$ is $x_1$. Then by Lemma \ref{lem:use-for-1-1-correspondence} we have that $1\in \dom(\psi)$. So $\psi(1)\in\Q$, and $\psi$ has  signature $(-m_1,\ldots,-m_{N-1})$.

To prove $x_1\in \dom(\psi)$ first observe that for $1\leq i\leq k-1$, 
$$\range(\psi_i^{-1})\subset\dom(\psi_{i+1}^{-1}).$$
Indeed, if $i$ is odd, then the orbits of $x_i$ under $\psi_i^{-1}$ and $\psi_{i+1}$ satisfy conditions of part (ii) of Corollary \ref{cor:dom-range-inclusion}. So $\range(\psi_i^{-1})\subset\range(\psi_{i+1})=\dom(\psi_{i+1}^{-1})$. On the other hand, if  $i$ is even, then the orbits of $x_i$ under $\psi_i^{-1}$ and $\psi_{i+1}$ satisfy conditions of part (i) of Corollary \ref{cor:dom-range-inclusion}. Therefore $\range(\psi_i^{-1})\subset\range(\psi_{i+1})=\dom(\psi_{i+1}^{-1})$, and we are done. 

Finally, suppose there are  legal functions $\phi$ and $\phi'$ with signature $(-m_1,\ldots,-m_{N-1})$ and $1\in\dom(\phi)\cap\dom(\phi')$. Then $\delta(\phi)=\delta(\phi')$, and by Corollary  \ref{cor:h-delta-increasing}, 
$\phi(1)=\phi'(1)$. This proves that the correspondence is one-to-one. 
The proof of the other side of the Lemma is analogous. 
\end{proof}


\subsection{Equivalence of intervals between points of $\P$ and $\Q$}
Some complications need to be dealt with before we can give the construction of $\pi$. For instance, the definition of $\pi$ on intervals $[x,y]$ and $[r^*_i(x),r^*_i(y)]$ are closely related. We therefore define an equivalence relation amongst certain intervals in $[0,1]$ to take such relations into account. Note that we are still in the settings of Assumption \ref{assump:suff}, but we cover both the case $\P=\Q$ and the case $\P\cap \Q=\emptyset$.

\begin{lemma}\label{lem:equivalent}
Let $w$ be as in  Assumption \ref{assump:suff}. 
 Let $\overline{\P\cup\Q}$ denote the closure of $\P\cup \Q$ in the usual topology of $[0,1]$. For a countable index 
set $I$ and pairwise disjoint open intervals $I_i$, we have $[0,1]\setminus (\overline{\P\cup\Q})=\cup_{i\in I}I_i$.   Let $\phi$ be a legal function.  Then,
\begin{itemize}
\item[(i)] If $\phi(I_i)\cap I_i\neq \emptyset$ then $\phi$ is the identity function on its domain. 
\item[(ii)] If $\phi(I_{i'})\cap I_{i}\neq \emptyset$ then $\phi(I_{i'})=I_{i}$.
\end{itemize}
\end{lemma} 
\begin{proof}
First, we prove that there are no relations between the points of an interval $I_i$, $i\in I$, {\it i.e.} if $x\in I_i$ and $\phi$ is a legal function with $\phi(x)\in I_i$, then $\phi$ must be the identity function on its domain. Note that each $I_i$ is an open interval in $[0,1]$. Let $a_i=\inf\{x:\ x\in I_i\}$ and $b_i=\sup\{x:\ x\in I_i\}$. Then $a_i$ and $b_i$ belong to $\overline{\P\cup\Q}$. 
Fix $i$. Suppose that $x\in (a_i,b_i)$, and $\phi$ is a legal function such that $\phi(x)\in (a_i,b_i)$. Towards a contradiction, assume that $\phi(x)\neq x$. Without loss of generality, suppose $\phi(x)>x$. By Proposition \ref{prop:orbits-domain}, we have $\dom(\phi)=[p,q]$ with $p\in\P, q\in\Q$ and $ [a_i,b_i]\subseteq [p,q]$. By Corollary \ref{cor:h-delta-increasing}, $\delta(\phi)>0$ and for every $z\in \dom(\phi)$ we have $z<\phi(z)$. In particular, $\phi^i(p)<\phi^{i+1}(p)$, whenever $p\in \dom(\phi^{i+1}).$ 

Let $M$ denote the positive integer such that $p\in \dom(\phi^M)$ and $p\not\in\dom(\phi^{M+1})$. First, note that such an integer exists. Indeed, if $p\in \dom(\phi^j)$ for every positive integer $j$, then the increasing sequence  $\{\phi^j(p)\}_{j\in{\mathbb N}}$ lies inside $[p,q]$. Therefore, $p_0=\lim_{j\rightarrow \infty} \phi^j(p)$ lies inside $[p,q]=\dom(\phi)$ as well, and we have $\phi(p_0)=p_0$. But this is a contradiction with $\delta(\phi)>0$.  Clearly, for $M$ as above, we have
$$p<\phi(p)<\phi^2(p)<\ldots <\phi^M(p),$$
with $p<a_i$ and $\phi^M(p)>q_i$. Since $\{p,\phi(p),\ldots, \phi^M(p)\}\subseteq \P$, none of these points  lie inside $(a_i,b_i)$. Thus, there exists $1\leq i_0\leq M-1$ with $\phi^{i_0}(p)\leq a_i$ and 
$\phi^{i_0+1}(p)\geq b_i$. So, $\phi^{i_0}(p)<x<\phi^{i_0+1}(p)$. Since $\phi$ is strictly increasing, this implies that $\phi^{i_0+1}(p)<\phi(x)$, and in particular, $\phi(x)\not\in (a_i,b_i)$ which is a contradiction. Thus, we must have $\delta (\phi)=0$, and $\phi$ is the identity function on its domain. This proves (i).

Now suppose that $\phi$ is a legal function. Let $x\in I_{i'}=(a_{i'},b_{i'})$ and $\phi(x)\in I_i=(a_i,b_i)$. In particular, we have $x\in \dom(\phi)$. Thus, there exist $p\in \P$ and $q\in \Q$ such that $I_{i'}\subseteq [p,q]=\dom(\phi)$, since 
$I_{i'}\cap (\P\cup\Q)=\emptyset$. Since $\phi$ is a strictly increasing continuous function we have that $\phi(a_{i'})<\phi(x)<\phi(b_{i'})$. In particular, $\phi(I_{i'})=(\phi(a_{i'}),\phi(b_{i'}))$. Also, we know that $a_i$, $b_i$, $\phi(a_{i'})$ and $\phi(b_{i'})$ are points in $\overline{\P\cup\Q}$. 
Since for every $j$, $I_j\cap(\overline{\P\cup\Q})=\emptyset$ and $\phi(I_{i'})\cap I_i\neq\emptyset$, then $\phi(a_{i'})= a_i$ and $\phi(b_{i'})= b_i$. Thus, 
$\phi(I_{i'})=I_i$.
\end{proof}

We say $i\sim i'$ if there exists a legal function $\phi$ such that $I_i\cap \phi(I_{i'})\neq \emptyset$ (or equivalently if $\phi(I_{i'})=I_i$). The relation $\sim$ is an equivalence relation. Consider the equivalence classes produced by $\sim$. For each $i$, we denote the equivalence class of $I_i$ by $[I_i]$. 
\subsection{Conditions of Theorem \ref{thm:nec-suff-cond-N} are sufficient.}
In this subsection we prove the sufficiency of the conditions of Theorem \ref{thm:nec-suff-cond-N} by constructing a uniform linear embedding $\pi:[0,1]\rightarrow \Rrr$
when a function $w$ satisfies the conditions of the theorem as given in Assumption \ref{assump:suff}.

Define $\pi$ first on $\P$ by $\pi(x)=\delta(\phi)$, where $\phi$ is a legal function with $\phi(0)=x$.
By Corollary \ref{cor:h-delta-increasing}, if there are two legal functions $\phi_1, \phi_2$ with $\phi_1(0)=\phi_2(0)=x$ then $\delta(\phi_1)=\delta(\phi_2)$. Thus $\pi$ is well-defined on $\P$. Moreover (2a) tells us that $\pi$ is strictly increasing on $\P$. 

Next, we will extend $\pi$ to a strictly increasing function on $\P\cup\Q$. If $\P=\Q$ then there is nothing to do. So assume that $\P\cap\Q=\emptyset$. Let 
\begin{eqnarray*}
m&=&\sup\{\delta(\phi)| \phi \mbox{ is a legal function and } 0\in\dom(\phi) \}\\
M&=&\min_{1\leq i\leq N-1}\, \inf\{d_i-\delta(\psi)|\, 1\in\dom(\psi), \psi(1)<r_i^*(0)\}
\end{eqnarray*}
By Condition (2b) of Theorem \ref{thm:nec-suff-cond-N},  $m\leq M$. Choose $\pi(1)$ as follows: If  $m=M$, then let $\pi^-(1)=\pi(1)=m$.  Otherwise, choose $\pi^-(1)=\pi(1)\in(m,M)$. Observe that
if $\phi$ and $\psi$ are legal functions with $0\in \dom(\phi)$, $1\in \dom(\psi)$, and $\psi(1)<r_i^*(0)$ then $\delta(\phi)<\pi(1)<d_i-\delta(\psi)$. This is clear when $m<M$. In the case where $m=M$, we have $\pi(1)=m=M=a$, and the desired inequality is given by Condition (2b).  


Now define $\pi$ on $\Q$ to be $\pi(y)=\pi(1)+\delta(\phi)$, where $\phi$ is a legal function with $y=\phi(1)$.
Let $y\in\Q$, and $\phi_1,\phi_2$ be legal functions with $\phi_1(1)=\phi_2(1)=y$. Then by Corollary \ref{cor:h-delta-increasing}, $\delta(\phi_2)=\delta(\phi_1)$. Also, note that $\P\cap\Q=\emptyset$, so the function $\pi$ as defined on $\Q$ is well-defined.  Moreover  $\pi$ is strictly increasing on $\Q$ by Lemma \ref{lem:delta-increasing}. 
\begin{claim}
$\pi$ is increasing on $\P\cup\Q$.
\end{claim}
\begin{proof}[Proof of claim]
Recall that we are assuming $\P\cap \Q=\emptyset$. We consider two cases:\\

\noindent{\bf Case 1:} Assume $x\in \P$, $y\in \Q$, $x<y$. Let $\phi$ and $\psi$ be legal functions with $\phi(0)=x$ and $\psi(1)=y$. By Proposition \ref{prop:orbits-domain}, $\dom(\psi)=[p,1]$ and $\range(\psi)=[\psi(p),\psi(1)]$ where $p\in\P$. If $x=\phi(0)\in \range(\psi)$, then $\psi^{-1}\circ \phi(0)\in\P$. Thus, by choice of $\pi(1)$ we have $\delta(\phi)-\delta(\psi)<\pi(1)$. This implies that $\pi(x)=\delta(\phi)<\pi(1)+\delta(\psi)=\pi(y)$ and we are done. Now assume that $x\notin\range(\psi)$. Then $x<\psi(p)$. By a similar argument, we have  $\pi(\psi(p))<\pi(y)$. Also, since $\pi$ is strictly increasing on $\P$, $\pi(x)<\pi(\psi(p))$ and thus $\pi(x)<\pi(y)$. This completes the proof of the claim for this case.\\

\noindent{\bf Case 2:} Assume $x\in \P$, $y\in \Q$, $y<x$. Let $\phi=f_1\circ\ldots\circ f_s$  be a legal composition with $\phi(0)=x$.  Since $0$ belongs to the 
domain of $\phi$, we know that $f_s$ is an upper boundary function, say $r^*_i$. Let $\phi_1:=f_1\circ\ldots\circ f_{s-1}$. If $y\in \range(\phi_1)= \dom(\phi_1^{-1})$ then $\phi_1^{-1}\circ \psi(1)<r_i^*(0)$, where $\psi$ is a legal function with $\psi(1)=y$. By the choice of $\pi(1)$ we have $\pi(1)<d_i-(\delta(\psi)+\delta(\phi_1^{-1}))$. Therefore, $\pi(1)+\delta(\psi)<d_i+\delta(\phi_1)=\delta(\phi)$ and thus $\pi(y)<\pi(x)$.

Let us now assume that $y\notin \dom(\phi_1^{-1})$, $\phi_1^{-1}=f_{s-1}^{-1}\circ\ldots\circ f_1^{-1}$. Thus there exists $0\leq t\leq s-2$ such that $\eta_t(y):=f_{t}^{-1}\circ\ldots\circ f_0^{-1}(y)\notin \dom(f_{t+1}^{-1})$ where $f_0$ is the identity function. Since $y<x$ and $x\in \dom(\phi^{-1})$, this implies that 
$f_{t+1}^{-1}$ must be a lower boundary function, say $\ell^*_j$. Also $\eta_t(y)<r_j^*(0)$. Hence by the choice of $\pi(1)$ we have $\pi(1)<d_j-(\delta(\psi)+\delta(\eta_t))$. On the other hand, since  $\eta_t(x)\in \dom(f_{t+1}^{-1})$, we have $r_j^*(0)\leq \eta_t(x)$. Therefore $d_j\leq \delta(\phi)+\delta(\eta_t)$, as $\pi$ is strictly increasing on $\P$. Thus, $\pi(1)+\delta(\psi)<d_j\leq \delta(\phi)$, {\it i.e.} $\pi(y)<\pi(x)$, and we are done.
\end{proof}

We now extend $\pi$ to $\overline{\P\cup\Q}$, and prove that this is a well-defined process, {\it i.e.} for $x\in\overline{\P\cup\Q}\setminus \P\cup\Q$ the value of $\pi(x)$ is independent of the choice of the sequence converging to $x$. 

Let $\{x_n\}$ be a sequence in $\P\cup\Q$ converging to a limit point $x$. Since $\pi$
is increasing on $\P\cup\Q$ we have, $\lim _{i\in I_1} \pi(x_i)=\sup_{i\in I_1}\pi(x_i)$ and $\lim_{j\in I_2}\pi(x_j)=\inf_{j\in I_2}\pi(x_j)$, where  $I_1$ (respectively  $I_2$) is the set of positive integers $i$ with $x_i\leq x$ (respectively $x_i> x$). Define $\pi(x)=\sup_{i\in I_1}\pi(x_i)$, if there exists a sequence in $\P\cup\Q$ converging to x from the left. Otherwise, define $\pi(x)=\inf_{j\in I_2}\pi(x_j)$. This extension is well-defined. 
Indeed, let $\{x_n\}$ and $\{y_n\}$ be two sequences in $\P\cup \Q$ which converge to $x\in[0,1]\setminus (\P\cup\Q)$. Without loss of generality, assume that $\{x_n\},\{y_n\}\in[0,x]$. Since $\pi$ is increasing, it is
easy to see that $\sup\{\pi(x_n):n\in\Nnn\}=\sup\{\pi(y_n):n\in\Nnn\}$. Using a similar argument for infimum, we conclude that the extension is well-defined. 
Moreover if a sequence $\{x_n\}$ in $\P\cup\Q$ converges to $x\in\overline{\P\cup\Q}\setminus \P\cup\Q$ from the left, then $\{\ell_j^*(x_n)\}$ (respectively $\{r_j^*(x_n)\}$) converges to $\ell_j^*(x)$(respectively $r_j^*(x)$) from the left as well. Therefore the function $\pi$ defined on $\overline{\P\cup\Q}\setminus \P\cup\Q$ satisfies properties of a uniform linear embedding, {\it i.e.} Conditions (\ref{eq:link-r}) and (\ref{eq:link-l}).

We now define $\pi$ on $[0,1]\setminus \overline{\P\cup\Q}$ using the equivalence relation defined in Lemma \ref{lem:equivalent}. First recall that 
$[0,1]\setminus\overline{\P\cup\Q}=\cup_{i\in I} I_i$ for a countable index set $I$. As discussed in the proof of Lemma \ref{lem:equivalent}, $I_i=(a_i,b_i)$ where $a_i,b_i\in \overline{\P\cup \Q}$.
For each equivalence class $[I_i]$, proceed as follows. First, pick a representative $I_i$ for $[I_i]$, and define $\pi$ on $I_i$ to be the linear function with $\pi(a_i)$ and $\pi(b_i)$ as defined earlier. Next, for every $I_j\in [I_i]$, let $\phi$ be a legal function such that $\phi(I_j)=I_i$, equivalently
 $a_i=\phi(a_j)$ and $b_i=\phi(b_j)$.
For every $x\in I_j=(a_j,b_j)$, we define $\pi(x)$ according to the definition of $\pi$ on $I_i$, {\it i.e.},
$$\pi(x)=\pi(\phi(x))-\delta(\phi).$$
Thus $\pi$ extends to a strictly increasing function on $[0,1]$ which gives us the desired uniform linear embedding. 

\subsection{Construction of $\pi$: examples}

In this section, we illustrate the construction of a uniform linear embedding with an example. 
We first show with an example that there are indeed diagonally increasing functions with uniform embedding which produce infinite  $\P\cup \Q$, even in the case of a three-valued function $w$.

\subsubsection*{Example 2} For $i=1,2$ let $r_i(x)=x+b_i$, where $0< b_1< b_2<\frac{1}{2}$. Moreover assume that $\frac{b_1}{b_2}$ is irrational. We produce a sequence $x_i\in \P_i$ inductively. Let $x_0=0$ and $x_1=r_2(0)=b_2$.
For each $i>1$, define 
$$x_i=\left\{
\begin{array}{cc}
r_2(x_{i-1})=x_{i-1}+b_2 & \mbox{ if } x_{i-1}<\frac{1}{2}  \\
\ell_1(x_{i-1})=x_{i-1}-b_1&   \mbox{ if } x_{i-1}\geq \frac{1}{2}  
\end{array}
\right.$$
Clearly, $x_i\in \P_i$ as it always lies in $(0,1)$. Also each $x_i$ is in the form of $m_ib_2-n_ib_1$ for positive integers $n_i,m_i$. Moreover, at each step we increase the value of either $m_i$ or $n_i$ 
by exactly 1. Thus, $m_i+n_i=i$ for every $i$. It is easy to observe that $x_i$'s are all distinct. Indeed, if $x_i=x_j$ for positive integers $i,j$, then $b_1(n_i-n_j)=b_2({m_i-m_j})$, which is a
contradiction.


The following example considers a diagonally increasing function $w$ with a finite set of constrained points. 

\subsubsection*{Example 3}
Let $w$ be a well-separated diagonally increasing $\{\alpha_1,\alpha_2,\alpha_3\}$-valued function with $\alpha_1>\alpha_2>\alpha_3$ and the following boundary functions.
\begin{figure}[ht]
\centerline{\includegraphics[width=0.45\textwidth]{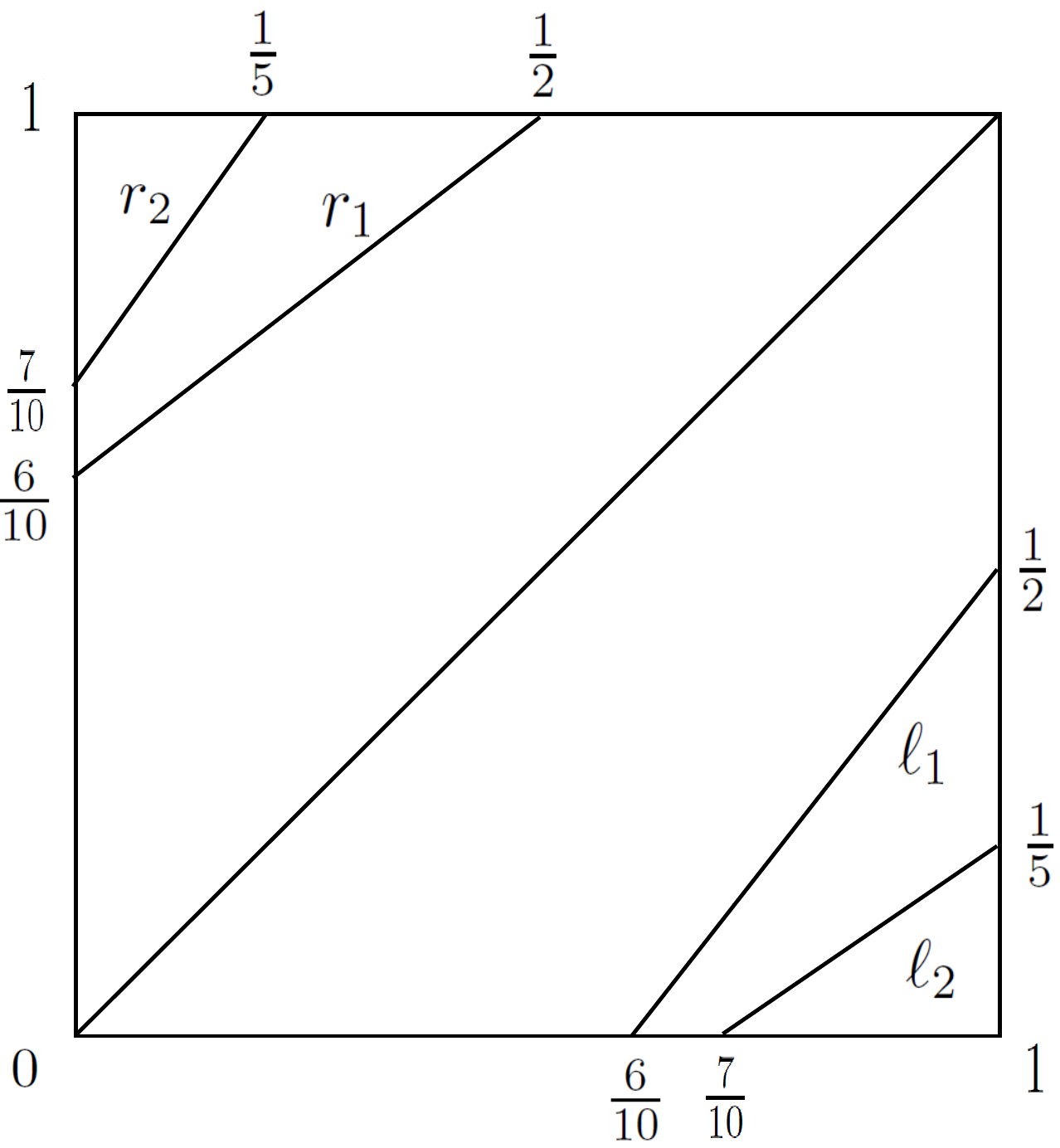}}
\end{figure}

$$r_1(x)=\left\{
\begin{array}{cc}
\frac{8}{10}x+\frac{6}{10}& x\in[0,\frac{1}{2}]  \\
  1   &  x\in [\frac{1}{2},1]\\
\end{array}
\right.,
\ell_1(x)=\left\{
\begin{array}{cc}
0 & x\in[0,\frac{6}{10}]  \\
\frac{10}{8}(x-\frac{6}{10})  &  x\in [\frac{6}{10},1]\\
\end{array}
\right.$$

$$r_2(x)=\left\{
\begin{array}{cc}
\frac{3}{2}x+\frac{7}{10}& x\in[0,\frac{1}{5}]  \\
  1   &  x\in [\frac{1}{5},1]\\
\end{array}
\right.,
\ell_2(x)=\left\{
\begin{array}{cc}
0 & x\in[0,\frac{7}{10}]  \\
\frac{2}{3}(x-\frac{7}{10})  &  x\in [\frac{7}{10},1]\\
\end{array}
\right.$$
{\bf Generate $\P$ and $\Q$:} We first find the set of constrained points of $w$.
\begin{equation*}
\P=\{0, r_1^*(0), r_2^*(0), \ell_1^*r_2^*(0), r_2^*\ell_1^*r_2^*(0), \ell_1^*r_2^*\ell_1^*r_2^*(0)\}\\
=\{0, \frac{6}{10}, \frac{7}{10}, \frac{1}{8}, \frac{71}{80}, \frac{23}{64}\}
\end{equation*}
\begin{equation*}
\Q=\{1, \ell_1^*(1), \ell_2^*(1), r_1^*\ell_2^*(1), \ell_2^*r_1^*\ell_2^*(1), r_1^*\ell_2^*r_1^*\ell_2^*(1)\}\\
=\{1, \frac{1}{2}, \frac{1}{5}, \frac{38}{50}, \frac{1}{25}, \frac{158}{250}\}
\end{equation*}

{\bf Check Condition (2) of Theorem \ref{thm:nec-suff-cond-N}}: We now check if Conditions (2a) and (2b) of Theorem \ref{thm:nec-suff-cond-N} holds. The order of elements of $\P$ is $0<\frac{1}{8}<\frac{23}{64}<\frac{6}{10}<\frac{7}{10}< \frac{71}{80}$. Hence Condition (2a) holds if there are real numbers $d_2>d_1>0$ such that

$$0<d_2-d_1<2d_2-2d_1<d_1<d_2<2d_2-d_1 $$ 

This gives us the following system of inequalities.

\begin{eqnarray*}
d_1-d_2&<&0\\
2d_2-3d_1&<&0
\end{eqnarray*}

$d_1=1$ and $d_2=\frac{5}{4}$ is a solution to the above system of inequalities.
Now let us check Condition (2b). We have 
\begin{eqnarray*}
\frac{1}{25}, \frac{1}{5}, \frac{1}{2}&<&\frac{6}{10}=r_1^*(0)\\
\frac{1}{25}, \frac{1}{5}, \frac{1}{2}, \frac{158}{250}&<&\frac{7}{10}=r_2^*(0)
\end{eqnarray*}

The system of inequalities obtained from the above inequalities is the following

\begin{eqnarray*}
d_1-d_2&<&0\\
2d_2-3d_1&<&0
\end{eqnarray*}

which is the same as the system of inequalities obtained from Condition (2a). Therefore $d_1=1$ and $d_2=\frac{5}{4}$ satisfy Conditions (2a) and (2b).

{\bf Construct $\pi$:} We now construct the function $\pi$ based on the proof of the Theorem \ref{thm:nec-suff-cond-N}. For $d_1=1$ and $d_2=\frac{5}{4}$, we have $m=\frac{6}{4}$ and $M=\frac{7}{4}$. Pick $\pi(1)=\frac{13}{8}$.
The equivalence classes of lemma \ref{lem:equivalent} for the function $w$ are:
\begin{eqnarray*}
[I_1]&=&\left\{(0,\frac{1}{25}), (\frac{1}{8},\frac{1}{5}), (\frac{23}{64},\frac{1}{2}), (\frac{6}{10},\frac{158}{250}), (\frac{71}{80},1)\right\},\\
\left[I_2\right]&=&\left\{(\frac{1}{25},\frac{1}{8}), (\frac{1}{5},\frac{23}{64}), (\frac{158}{250},\frac{7}{10}), (\frac{38}{50},\frac{71}{80})\right\},\\
\left[I_3\right]&=&\left\{(\frac{1}{2},\frac{6}{10})\right\}.\\
\end{eqnarray*}
$$\pi(x)=\left\{
\begin{array}{cc}
\frac{25}{8}x& x\in[0,\frac{1}{25}]\\
\frac{25}{17}(x-\frac{1}{8})+\frac{1}{4}&x\in[\frac{1}{25},\frac{1}{8}]\\
\frac{5}{3}x+\frac{1}{24}& x\in[\frac{1}{8},\frac{1}{5}]\\
\frac{5}{51}(8x-\frac{23}{8})+\frac{1}{2}& x\in[\frac{1}{5},\frac{23}{64}]\\
\frac{1}{9}(8x-1)+\frac{7}{24}& x\in[\frac{23}{64},\frac{1}{2}]\\
\frac{15}{4}(x-\frac{6}{10})+1& x\in[\frac{1}{2},\frac{6}{10}]\\
\frac{250}{64}(x-\frac{6}{10})+1& x\in[\frac{6}{10},\frac{158}{250}]\\
\frac{25}{136}(10x-7)+\frac{5}{4}& x\in[\frac{158}{250},\frac{7}{10}]\\
\frac{50}{24}(x-\frac{7}{10})+\frac{5}{4}& x\in[\frac{7}{10},\frac{38}{50}]\\
\frac{25}{51}(2x-\frac{71}{40})+\frac{3}{2}& x\in[\frac{38}{50},\frac{71}{80}]\\
\frac{10}{9}(x-\frac{7}{10})+\frac{31}{24}& x\in[\frac{71}{80},1]\\
\end{array}
\right .
$$

\section*{Acknowledgements}
The first author acknowledges the support of the Killam fellowship, and the third author acknowledges
support of NSERC.
This work was done during the second author's visits to Dalhousie University,  in June 2014 and May 2015. She thanks the Department of Mathematics and Statistics at Dalhousie for their hospitality.




\end{document}